\documentclass[reqno]{amsart}            
\usepackage{amsthm,amssymb,amsfonts,graphicx,cite,color}
\usepackage{dsfont}
\usepackage[bookmarks,bookmarksnumbered,bookmarksopen,colorlinks,backref,linkcolor=blue,citecolor=red]{hyperref}%
\usepackage{xcolor}
\usepackage{graphicx}
\usepackage{mathrsfs}

\newtheorem{theorem}{Theorem}[section]
\newtheorem{definition}{Definition}[section]
\newtheorem{lemma}{Lemma}[section]
\newtheorem{corollary}{Corollary}[section]
\newtheorem{remark}{Remark}[section]

\numberwithin{equation}{section}

\begin{document}

\title[Asymptotic expansion for the solution of  a convection-diffusion problem]
{Asymptotic expansion for the solution of  a convection-diffusion problem in a thin graph-like junction
}
\author[T.A.~Mel'nyk and A.V. Klevtsovskiy]{ Taras A. Mel'nyk and Arsen V. Klevtsovskiy}
\address{\hskip-12pt  Faculty of Mathematics and Mechanics, Department of Mathematical Physics\\
Taras Shevchenko National University of Kyiv\\
Volodymyrska str. 64,\ 01601 Kyiv,  \ Ukraine
}
\email{melnyk@imath.kiev.ua, \ \ avklevtsovskiy@gmail.com}

\begin{abstract}
A steady-state  convection-diffusion problem with a small diffusion of order $\mathcal{O}(\varepsilon)$
 is considered  in a thin three-dimensional graph-like junction consisting  of  thin  cylinders connected through a domain (node) of diameter $\mathcal{O}(\varepsilon),$ where $\varepsilon$ is a small parameter.
 Using multiscale analysis, the asymptotic expansion  for the solution is constructed and justified. The asymptotic estimates in the norm of Sobolev space $H^1$ as well as  in the uniform norm are proved for the difference between the solution and proposed approximations  with a predetermined accuracy with respect to the degree of $\varepsilon$.
\end{abstract}

\keywords{Asymptotic expansion,  convection-diffusion problem,   asymptotic estimate, thin star-shaped junction.
\\
\hspace*{9pt} {\it MOS subject classification:} \  76R99, 35J25, 35B40, 35B25, 74K30
}

\maketitle
\tableofcontents

\section{Introduction}

{\bf 1.} Convection-diffusion problems have recently attracted a lot of attention in the research literature (see e.g. recent monograph \cite{Stynes-2018} and references there). In  \cite{Mor96}  many applied problems, where convection-diffusion equations involve, are indicated; these are both the drift-diffusion equations for modeling semiconductor devices and the Black - Scholes equation from financial modeling. Here one can also mention heat transfer and chemical concentration models, image processing in computer graphics and medical applications, simulations of multiphase flows and  material crystallizations involving convection-diffusion equations. In  \cite{Mor96} it also noted that \textit{"accurate modeling of the interaction between convective and diffusion processes is the most common and difficult problem in the numerical approximation of partial differential equations."}

In general form a steady-state convection-diffusion equation is as follows:
$$
-  \mathrm{div} \big( D\, \nabla u\big) +  \mathrm{div} \big( \overrightarrow{V} u \big) = f,
$$
where $u$  is the concentration for mass transfer, $D$ is the  diffusion matrix,
$\overrightarrow{V}= (v_1, v_2, v_3)$ is the velocity field that the quantity is moving with, $f$ describes sources of the quantity $u,$
$\nabla u$ is the gradient of $u.$
If the velocity field is described as an incompressible flow, i.e., $\mathrm{div}\overrightarrow{V} = \sum_{i=1}^{3} \frac{\partial v_i}{\partial x_i}= 0,$ then the equation is simplified:  $-  \mathrm{div} \big( D\, \nabla u\big) +  \overrightarrow{V}\cdot  \nabla u = f.$

In the case when the P\'eclet number is large, that is, when diffusion is very small, then there is  a small parameter $\varepsilon$ near the diffusion terms
$$
 - \varepsilon\,  \mathrm{div} \big( D\, \nabla u\big) +  \mathrm{div} \big( \overrightarrow{V} u \big) = f
$$
and the convection terms  have a dominant  and strong influence on the solution  when  $\varepsilon$ is close to zero.
Obviously, that the thorough justification both numerical approximations and asymptotic approximations for solutions to such convection-dominated problems is more hard task. The reason is that the corresponding solutions have boundary layers, i.e., there are narrow areas where solutions are bounded regardless of $\varepsilon$, but where their derivatives are large \cite[\S 1.1]{Stynes-2018}.

In the majority of works devoted to error estimates for numerical approximations of convection-dominated problems with the Dirichlet boundary conditions, global estimates are proved in the weighted energy norm $\big(\varepsilon \|\nabla u\|^2_{L^2(\Omega)} + \| u\|^2_{L^2(\Omega)}\big)^{1/2}$ (see e.g. \cite[\S 4.2.1]{Stynes-2018}, \cite{Verfuerth_2005}).
But, as was noted in \cite{Jorn_Knob_Novo_2018} these bounds are not robust and
in \cite[\S 4.2]{Stynes-2018} one points out that \textit{"It is in general difficult to derive sharp bounds on derivatives of solutions of convection-diffusion problems inside characteristic boundary and interior layers."}  Error estimates for convection-dominated problems with mixed boundary conditions are missing.

\medskip

{\bf 2.} Our  work deals with the construction of the complete asymptotic expansion for the solution to a
steady-state  convection-dominated  problem in a thin three-dimensional graph-like junction consisting  of three  thin  cylinders connected through a domain (node) of diameter $\mathcal{O}(\varepsilon).$ In addition, we consider the Dirichlet boundary conditions on the bases
of the thin cylinders and the perturbed Neumann type boundary condition
$$
\big(-  \varepsilon \,  D\nabla u +  u \, \overrightarrow{V}\big) \cdot \boldsymbol{\nu}   =  \varepsilon \varphi
$$
on the lateral surfaces, in which the flows of both the convective and diffusion parts are involved.
 The  aim  is to gain a fundamental understanding of how surface interactions  influence transport processes in different  media, including porous ones (a material consisted of solid and voids (cavities, channels or interstices)).

To construct the asymptotic expansion  for the solution, we use the approach developed in our papers \cite{Klev_2019,Mel_Klev_AA-2016,M-AA-2021}. Here we have adapted this method to non-self-adjoint boundary value problems with a small parameter at the highest derivatives. The expansion consists of 3 parts, namely, the regular part of the
asymptotics located inside of each thin cylinder,  the inner part of the asymptotics discovered in a neighborhood of the
node, and the boundary-layer asymptotics near the base of one thin cylinder. The principal new feature and novelty  of this paper in comparison with the papers mentioned above are the following two points.

The terms of the inner part of the asymptotics are special solutions of boundary-value problems in an unbounded
domain with different outlets at infinity. And if in the previous problems considered in  \cite{Klev_2019,Mel_Klev_AA-2016,M-AA-2021} such solutions had  polynomial growth at infinity, now they should stabilize at each outlet to some constant. The necessary and sufficient conditions for the existence of such solutions are proved in \S~\ref{subsec_Inner_part}, moreover, it is shown that this stabilization is exponential.

Secondly, an a priori estimate for the solution of this convection-dominated problem in the usual energy norm of the Sobolev space $H^1$ is proved in Theorem~\ref{apriory_estimate} under additional reasonable assumptions for  the velocity field  $\overrightarrow{V}.$ With the help of this estimate we justify the constructed asymptotics and prove estimates for the difference between the solution of the original problem and asymptotic approximations with a predetermined accuracy with respect to the degree of $\varepsilon.$

It should be stressed that the error estimates are very important both for justification of adequacy of one-dimensional models that aim at description of actual three-dimensional thin bodies and for the study of boundary effects and effects of local (internal) inhomogeneities in applied problems. We also hope that the a priori estimate proved here will help to construct robust numerical approximations for convection-dominated problems.

\medskip

{\bf 3.} Investigations of various physical and biological processes in thin channels, junctions, and networks are urgent for numerous fields of natural sciences (see e.g. \cite{Pan_2005,Post-2012}). We refer the reader to \cite{Mel_Klev_AA-2016} for the review of different asymptotic approaches to study boundary-value problems in thin graph-like domains. Here we mention some papers  connected to convection-diffusion problems in thin tubular structures.

An asymptotic expansion is constructed with a draft justification for the solution to a convection-diffusion equation
$$
 - \mathrm{div} \big( K\, \nabla u\big) +   \overrightarrow{V} \cdot \nabla u  = g,
$$
with  very specific structures of  the diffusion  and the  velocity field in a $2D$-thin tubular structure in \cite{Car_Pan_Sir_2010}; the Neumann  condition, which involved only flux of the diffusion part, is  posed  on the lateral boundary and  inflow and outflow concentrations are given at the ends of thin bars.

In  \cite{Pan_Pan_Pia_2010} the leading terms of the asymptotics are constructed for the solution to a stationary convection-diffusion equation defined in a thin cylinder that is the union of two nonintersecting cylinders with a junction at the origin.
In each of these cylinders the coefficients are rapidly oscillating functions that are periodic in the axial direction, and the microstructure period is of the same order as the cylinder diameter. On the lateral boundary of the cylinder the Neumann boundary condition, which involved only flux of the diffusion part, is considered, while the Dirichlet boundary conditions are posed at the cylinder bases.

Convection-diffusion problems can be obtained from linearization Navier-Stokes equations with a large Reynolds number.
The method of the partial asymptotic domain decomposition was applied to study  the Navier-Stokes flow in thin  tube structures in \cite{P-P-Stokes-1-2015,P-P-Stokes-2-2015}.  Asymptotic analysis  of the steady incompressible flow of a Bingham fluid in  a thin T-like shaped plane domain, under the action of given external forces and with no-slip boundary condition on the whole boundary of the domain,  is made in \cite{Bun_Gau_Leo_2019}.

\medskip

{\bf 4.} After the statement of the problem in Sec. ~\ref{Sec:Statement}, the paper is organized as follows.  Section~\ref{Sec:expansions} deals with the construction formal asymptotic expansions for the solution to the problem \eqref{probl} in different parts of the thin graph-like junction.  In particular, as was noted above, we introduce a special inner asymptotic expansion in a neighborhood of the node, determine its coefficients and study some their properties as solutions to corresponding boundary-value problems in an unbounded
domain in \S~\ref{subsec_Inner_part}.  Writing down the solvability conditions for such problems, Kirchhoff conditions are derived at the vertex of the graph for terms of the regular expansion, in particular, the limit problem on the graph is obtained in \S~\ref{sub_limit_problem}.  In Sec.~\ref{Sec:justification}, we construct a complete asymptotic expansion in the whole thin graph-like junction and  prove the corresponding  asymptotic estimates.  In Conclusion the obtained results are analyzed and  research perspectives are considered.

\section{Statement of the problem}\label{Sec:Statement}
The model thin graph-like junction  $\Omega_\varepsilon$  consists of three thin  cylinders
$$
\Omega_\varepsilon^{(i)} =
  \Bigg\{
  x=(x_1, x_2, x_3)\in\Bbb{R}^3 \ : \
  \varepsilon \ell_0<x_i<\ell_i, \quad
  \sum\limits_{j=1}^3 (1-\delta_{ij})x_j^2<\varepsilon^2 h_i^2
  \Bigg\}, \quad i \in \{1, 2, 3\},
$$
that are joined through a domain $\Omega_\varepsilon^{(0)}$ (referred in the sequel "node").
Here $\varepsilon$ is a small parameter; $\ell_0\in(0, \frac13), \ \ell_i\geq1, \ i \in \{1, 2, 3\};$
 the constant $\{h_i\}_{i=1}^3$ are positive;
the symbol $\delta_{ij}$ is the Kroneker delta, i.e.,
$\delta_{ii} = 1$ and $\delta_{ij} = 0$ if $i \neq j.$

Denote the lateral surface of the thin cylinder $\Omega_\varepsilon^{(i)}$ by
$$
{\Gamma_\varepsilon^{(i)}} := \partial\Omega_\varepsilon^{(i)} \cap \{ x\in\Bbb{R}^3 \ : \ \varepsilon \ell_0<x_i<\ell_i \}
$$
and by
$$
\Upsilon_\varepsilon^{(i)} (y_i) := \Omega_\varepsilon^{(i)} \cap
\big\{  x\in\Bbb{R}^3 \, : \ x_i= y_i\big\}
$$
its cross-section at the point $y_i \in [ \varepsilon \ell_0, \ell_i].$

The node $\Omega_\varepsilon^{(0)}$ (see Fig.~\ref{f2}) is formed by the homothetic transformation with coefficient $\varepsilon$ from a bounded domain $\Xi^{(0)}\subset \Bbb R^3$,  i.e.,
$
\Omega_\varepsilon^{(0)} = \varepsilon\, \Xi^{(0)}.
$
In addition, we assume that its boundary contains the disks
$\Upsilon_\varepsilon^{(i)} (\varepsilon\ell_0), \ i \in \{1, 2, 3\},$
and denote
$
\Gamma_\varepsilon^{(0)} :=
\partial\Omega_\varepsilon^{(0)} \backslash
\left\{
 \overline{\Upsilon_\varepsilon^{(1)} (\varepsilon \ell_0)} \cup
 \overline{\Upsilon_\varepsilon^{(2)} (\varepsilon \ell_0)} \cup
 \overline{\Upsilon_\varepsilon^{(3)} (\varepsilon \ell_0)}
\right\}.
$

\begin{figure}[htbp]
\begin{center}
\includegraphics[width=5cm]{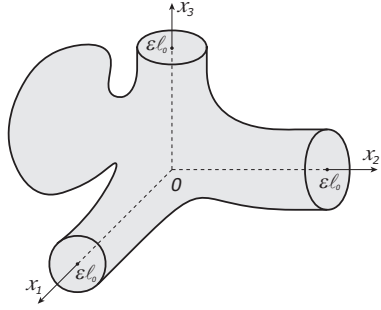}
\end{center}
\caption{The node $\Omega_\varepsilon^{(0)}$}\label{f2}
\end{figure}

Thus the model thin graph-like junction  $\Omega_\varepsilon$  (see Fig.~\ref{f3})
is   the interior of the union
$
\bigcup_{i=0}^{3}\overline{\Omega_\varepsilon^{(i)}}
$
and we assume that the surface $\partial\Omega_\varepsilon\setminus \bigcup_{i=0}^{3}\overline{\Upsilon_\varepsilon^{(i)} (\ell_i)}$ is smooth.

\begin{figure}[htbp]
\begin{center}
\includegraphics[width=8cm]{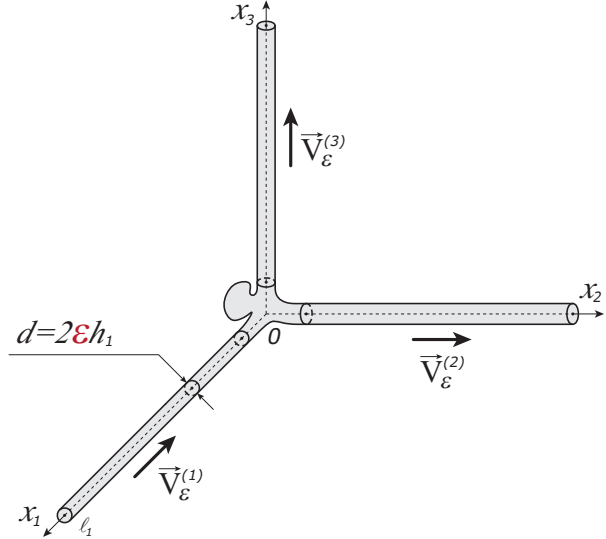}
\end{center}
\caption{The model thin graph-like  junction $\Omega_\varepsilon$}\label{f3}
\end{figure}

\begin{remark}
We can consider more general thin graph-like junctions with arbitrary orientation of thin cylinders and their number can be also arbitrary (see \cite{Klev_2019}).
 But to avoid technical and huge calculations and to demonstrate the main steps of the proposed asymptotic approach
 we consider the model thin   junction $\Omega_\varepsilon,$ in which the cylinders are placed along the coordinate axes.
\end{remark}

The given vector-valued function $\overrightarrow{V_\varepsilon}$ depends on the parts of $\Omega_\varepsilon$ and has the following structure:
\begin{gather}
\overrightarrow{V_\varepsilon}(x)=
\left(v^{(0)}_1(\tfrac{x}{\varepsilon}),  \  v^{(0)}_2(\tfrac{x}{\varepsilon}), \  v^{(0)}_3(\tfrac{x}{\varepsilon})
\right) =: \overrightarrow{V_\varepsilon}^{(0)}(x), \qquad x \in \ \Omega_\varepsilon^{(0)}, \notag
   \\
   \label{str_1}
\overrightarrow{V_\varepsilon}(x)=
\left(v^{(1)}_1(x_1), \ \varepsilon\,  v^{(1)}_2(x_1, \tfrac{\overline{x}_1}{\varepsilon}), \
 \varepsilon \, v^{(1)}_3(x_1,\tfrac{\overline{x}_1}{\varepsilon})
\right) =: \overrightarrow{V_\varepsilon}^{(1)}(x), \qquad x \in \ \Omega_\varepsilon^{(1)} \quad (v^{(1)}_1 < 0),
    \\
    \label{str_2}
\overrightarrow{V_\varepsilon}(x) =
\left( \varepsilon\,  v^{(2)}_1(x_2, \tfrac{\overline{x}_2}{\varepsilon}), \   v^{(2)}_2(x_2),  \
 \varepsilon \, v^{(2)}_3(x_2,\tfrac{\overline{x}_2}{\varepsilon})
\right) =: \overrightarrow{V_\varepsilon}^{(2)}(x), \qquad x \in \ \Omega_\varepsilon^{(2)} \quad (v^{(2)}_2 >  0),
     \\
     \label{st_3}
\overrightarrow{V_\varepsilon}(x) =
\left(\varepsilon\,  v^{(3)}_1(x_3, \tfrac{\overline{x}_3}{\varepsilon}), \
    \varepsilon \, v^{(3)}_2(x_3,\tfrac{\overline{x}_3}{\varepsilon}), \ v^{(3)}_3(x_3)
\right) =: \overrightarrow{V_\varepsilon}^{(3)}(x), \qquad x \in \ \Omega_\varepsilon^{(3)} \quad (v^{(3)}_3 >  0).
  \end{gather}
Here
$$
\overline{x}_i =
\left\{\begin{array}{lr}
(x_2, x_3), & i=1, \\
(x_1, x_3), & i=2, \\
(x_1, x_2), & i=3.
\end{array}\right.
$$
All components of the vector-valued function $\overrightarrow{V_\varepsilon}$ are smooth and bounded in $\overline{\Omega}_\varepsilon.$
In addition, the functions $v^{(1)}_1, v^{(2)}_2, v^{(3)}_3$ are equal to constants $\mathrm{v}_1, \mathrm{v}_2, \mathrm{v}_3$ in
 neighbourhoods  of the origin, respectively; the other components of the vector-functions $\overrightarrow{V_\varepsilon}^{(1)},$ $\overrightarrow{V_\varepsilon}^{(2)},$ $\overrightarrow{V_\varepsilon}^{(3)}$  have compact supports
with respect to the corresponding longitudinal  variable $x_i$ in $\Omega^{(i)}_\varepsilon$ $(i\in\{1,2,3\});$ and
\begin{gather}
\overrightarrow{V_\varepsilon}^{(0)}\big|_{x_1=\varepsilon \, \ell_0-0} = \overrightarrow{V_\varepsilon}^{(1)}\big|_{x_1=\varepsilon \, \ell_0+0} = \left(
\mathrm{v}_1, \, 0,\, 0\right) \ \  \text{on} \ \Upsilon_\varepsilon^{(1)} (\varepsilon\ell_0), \notag
\\
\overrightarrow{V_\varepsilon}^{(0)}\big|_{x_2=\varepsilon \, \ell_0-0} = \overrightarrow{V_\varepsilon}^{(2)}\big|_{x_2=\varepsilon \, \ell_0+0} = \left(0,\, \mathrm{v}_2, \, 0\right) \ \  \text{on} \ \Upsilon_\varepsilon^{(2)} (\varepsilon\ell_0), \notag
 \\ \label{v_i}
\overrightarrow{V_\varepsilon}^{(0)}\big|_{x_3=\varepsilon \, \ell_0-0} = \overrightarrow{V_\varepsilon}^{(3)}\big|_{x_3=\varepsilon \, \ell_0+0} =\left(
0,\, 0, \, \mathrm{v}_3\right) \ \  \text{on} \ \Upsilon_\varepsilon^{(3)} (\varepsilon\ell_0).
\end{gather}

These assumptions are natural and mean than the velocity field $\overrightarrow{V_\varepsilon}$ enters in the thin cylinder
$\Omega_\varepsilon^{(1)}$ and outgoes from the other cylinders $\Omega_\varepsilon^{(2)}$ and $\Omega_\varepsilon^{(3)}.$


In $\Omega_\varepsilon,$ we consider the following the convection-diffusion  problem:
\begin{equation}\label{probl}
\left\{\begin{array}{rcll}
  -  \varepsilon\, \mathrm{div} \big( \mathbb{D}^{(i)}_\varepsilon(x) \nabla u_\varepsilon\big) +
  \mathrm{div} \big( \overrightarrow{V_\varepsilon}^{(i)}(x) \, u_\varepsilon\big)
  & = & 0, &
    x \in\Omega_\varepsilon^{(i)}, \ \   i\in\{1,2,3\},
\\[2mm]
   -  \varepsilon \, \sigma_\varepsilon(u_\varepsilon)  + \,  u_\varepsilon\,  \overrightarrow{V_\varepsilon}^{(i)} \cdot \boldsymbol{\nu}
  & = & \varepsilon\, \varphi^{(i)}(x_i) &
   \text{on} \ \Gamma_\varepsilon^{(i)}, \ \   i\in\{1,2,3\},
\\[2mm]
-  \varepsilon\, \mathrm{div} \big( \mathbb{D}^{(0)}_\varepsilon(x) \nabla u_\varepsilon\big) +
  \mathrm{div} \big(\overrightarrow{V_\varepsilon}^{(0)}(x) \,  u_\varepsilon\big)  & = & 0, &
    x \in\Omega_\varepsilon^{(0)},
\\[2mm]
 -  \varepsilon \, \sigma_\varepsilon(u_\varepsilon) +  u_\varepsilon \, \overrightarrow{V_\varepsilon}^{(0)} \cdot \boldsymbol{\nu}  & = & 0  &
    \text{on} \ \Gamma_\varepsilon^{(0)},
\\[2mm]
 u_\varepsilon \big|_{x_i= \ell_i}
 & = & q_i, & \text{on} \ \Upsilon_{\varepsilon}^{(i)} (\ell_i) , \ \  i\in\{1,2,3\},
 \\[2mm]
 u_\varepsilon\big|_{x_i= \varepsilon \ell_0-0} &=& u_\varepsilon\big|_{x_i= \varepsilon \ell_0+0}, &  i\in\{1,2,3\},
 \\[2mm]
 a_{ii}^{(i)} \, \partial_{x_i}u_\varepsilon\big|_{x_i= \varepsilon \ell_0+0} &=&
 \sum_{j=1}^{3 } \big(a_{ij}^{(0)}(\frac{x}{\varepsilon}) \partial_{x_j} u_\varepsilon\big)\big|_{x_i= \varepsilon \ell_0-0}, &
 i\in\{1,2,3\},
\end{array}\right.
\end{equation}
where $\sigma_\varepsilon(u):= (\mathbb{D}_\varepsilon  \nabla u) \cdot {\boldsymbol{\nu}}$  is the conormal derivative of $u;$  ${\boldsymbol{\nu}}$ is the outward unit normal to $\partial \Omega_\varepsilon;$ the given constants $q_1, q_2, q_3$ are positive; \,  for each $i\in \{1,2,3\}$ the given function $\varphi^{(i)}$ belongs to the space $C_0^\infty((0, \ell_i)).$     Also, the diffusion matrix $\mathbb{D}_\varepsilon$  depends on the parts of the thin junction $\Omega_\varepsilon,$ namely,
$$
\mathbb{D}_\varepsilon(x) =
\left(
\begin{matrix}
  a^{(0)}_{11}(\frac{x}{\varepsilon}) & a^{(0)}_{12}(\frac{x}{\varepsilon}) & a^{(0)}_{13}(\frac{x}{\varepsilon}) \\[2mm]
  a^{(0)}_{21}(\frac{x}{\varepsilon}) & a^{(0)}_{22}(\frac{x}{\varepsilon}) & a^{(0)}_{23}(\frac{x}{\varepsilon}) \\[2mm]
  a^{(0)}_{31}(\frac{x}{\varepsilon}) & a^{(0)}_{32}(\frac{x}{\varepsilon}) & a^{(0)}_{33}(\frac{x}{\varepsilon})
\end{matrix}
\right)  =: \mathbb{D}_\varepsilon^{(0)}(x), \qquad x \in \ \Omega_\varepsilon^{(0)},
$$

$$
\mathbb{D}_\varepsilon(x)  =
\left(
\begin{matrix}
  a^{(1)}_{11} & 0 & 0 \\[2mm]
  0 & a^{(1)}_{22}(\frac{\overline{x}_1}{\varepsilon}) &  a^{(1)}_{23}(\frac{\overline{x}_1}{\varepsilon}) \\[2mm]
  0 &  a^{(1)}_{32}(\frac{\overline{x}_1}{\varepsilon}) &  a^{(1)}_{33}(\frac{\overline{x}_1}{\varepsilon})
\end{matrix}
\right)
 =:  \mathbb{D}_\varepsilon^{(1)}(x), \qquad x \in \ \Omega_\varepsilon^{(1)},
$$

$$
\mathbb{D}_\varepsilon(x)  =
\left(
\begin{matrix}
  a^{(2)}_{11}(\frac{\overline{x}_2}{\varepsilon}) & 0 & a^{(2)}_{13}(\frac{\overline{x}_2}{\varepsilon}) \\[2mm]
  0 & a^{(2)}_{22} & 0 \\[2mm]
  a^{(2)}_{31}(\frac{\overline{x}_2}{\varepsilon}) & 0 & a^{(2)}_{33}(\frac{\overline{x}_2}{\varepsilon})
\end{matrix}
\right)
 =:  \mathbb{D}_\varepsilon^{(2)}(x), \qquad x \in \ \Omega_\varepsilon^{(2)},
$$

$$
\mathbb{D}_\varepsilon(x) =
\left(
\begin{matrix}
  a^{(3)}_{11}(\frac{\overline{x}_3}{\varepsilon}) & a^{(3)}_{12}(\frac{\overline{x}_3}{\varepsilon}) & 0 \\[2mm]
  a^{(3)}_{21}(\frac{\overline{x}_3}{\varepsilon}) & a^{(3)}_{22}(\frac{\overline{x}_3}{\varepsilon}) & 0 \\[2mm]
  0 & 0 & a^{(3)}_{33}
\end{matrix}
\right)
  =:  \mathbb{D}_\varepsilon^{(3)}(x), \qquad x \in \ \Omega_\varepsilon^{(3)}.
$$
In addition, all matrices are symmetric, i.e., for all $i\in\{0, 1,2,3\}$  we have $a^{(i)}_{mn} = a^{(i)}_{nm},$ and   there exist positive constants $\kappa_0^{(i)}, \ \kappa_1^{(i)},$ and
$\varepsilon_0$  such that  for all $x \in \Omega_\varepsilon,$ $\varepsilon \in (0, \varepsilon_0),$
and $y \in \Bbb R^3$
\begin{equation}\label{n1}
  \kappa_0^{(i)} \, |y|^2 \le \sum_{m,n}^{3} a^{(i)}_{mn}(x,\varepsilon) y_m y_n \le \kappa_1^{(i)} \, |y|^2 ,
\end{equation}
and  their elements are smooth in all variables and bounded in $\overline{\Omega}_\varepsilon.$  From  \eqref{n1} it follows that the constants   $\{a^{(i)}_{ii}\}_{i=1}^3$ are positive.

For any functions $u, v \in C^2(\overline{\Omega}_\varepsilon),$
Green's formula hold, i.e.
\begin{equation}\label{G_1}
 \int_{\Omega_\varepsilon} \mathcal{L}(u) \, v \, dx
  =
 \varepsilon  \int_{\Omega_\varepsilon}\nabla v \cdot (\mathbb{D}_\varepsilon \nabla u) \, dx -
  \int_{\Omega_\varepsilon} u \, \overrightarrow{V_\varepsilon}\cdot \nabla v \, dx
  + \int_{\partial \Omega_\varepsilon}\big(-\varepsilon \sigma_\varepsilon(u)  + u\, \overrightarrow{V_\varepsilon}\cdot \boldsymbol{\nu} \big) \, v \, dS_x
\end{equation}
and
\begin{equation}\label{G_2}
 \int_{\Omega_\varepsilon}  \mathcal{L}^*(v)\, u \, dx
  =
 \varepsilon  \int_{\Omega_\varepsilon}\nabla u \cdot (\mathbb{D}_\varepsilon \nabla v) \, dx -
  \int_{\Omega_\varepsilon} u \, \overrightarrow{V_\varepsilon}\cdot \nabla v \, dx
  -\varepsilon  \int_{\partial \Omega_\varepsilon}\sigma_\varepsilon(v)   \, u \, dS_x ,
\end{equation}
where $ \mathcal{L}_\varepsilon(u) :=  -\varepsilon\, \mathrm{div}(\mathbb{D}_\varepsilon \, \nabla u) +  \mathrm{div}(\overrightarrow{V_\varepsilon}\, u)$ and $\mathcal{L}^*_\varepsilon(v):= -\varepsilon\, \mathrm{div}(\mathbb{D}_\varepsilon \, \nabla v) -  \overrightarrow{V_\varepsilon}\cdot \nabla v$.
Formulae \eqref{G_1} and \eqref{G_2}  mean that the operator $\mathcal{L}^*_\varepsilon$ is  formal adjoint of
$ \mathcal{L}_\varepsilon.$

\begin{definition}
A function $u_\varepsilon \in H^1(\Omega_\varepsilon)$ such that $u_\varepsilon|_{x_i=\ell_i} = q_i$ for $i\in \{1, 2, 3\}$ is called a weak solution to the problem \eqref{probl} if
\begin{equation}\label{int_def}
  \varepsilon \int_{\Omega_\varepsilon} \nabla \psi \cdot (\mathbb{D}_\varepsilon \nabla u_\varepsilon)\, dx -
  \int_{\Omega_\varepsilon} u_\varepsilon \, \overrightarrow{V_\varepsilon} \cdot \nabla \psi \, dx
      + \varepsilon \sum_{i=1}^{3}\int_{\Gamma_\varepsilon^{(i)}}\varphi^{(i)}\, \psi \, dS_x = 0
\end{equation}
for any function $\psi \in H^1(\Omega_\varepsilon; \Upsilon_\varepsilon):= \{u\in H^1(\Omega_\varepsilon)\colon u|_{\Upsilon_\varepsilon}=0\},$ where $\Upsilon_\varepsilon := \bigcup_{i=1}^3 \Upsilon_{\varepsilon}^{(i)}(\ell_i).$
\end{definition}

It follows from the theory of linear elliptic boundary-value problems that, for any fixed value of $\varepsilon,$ the problem  \eqref{probl}
possesses a unique weak solution. Indeed, if there are two solutions to the problem \eqref{probl} then their difference, the  function $\Phi \in H^1(\Omega_\varepsilon; \Upsilon_\varepsilon),$ is  a weak solution of
\begin{equation}\label{probl_Phi}
\left\{
  \begin{array}{l}
    \mathcal{L}_\varepsilon(\Phi) = 0 \quad \text{in} \ \Omega_\varepsilon,
 \\
    \Phi = 0\quad  \text{on} \ \Upsilon_\varepsilon, \qquad -\varepsilon \sigma_\varepsilon(\Phi)  + \Phi \, \overrightarrow{V_\varepsilon}\cdot \boldsymbol{\nu}
=0 \quad \text{on} \ \partial\Omega_\varepsilon\setminus \Upsilon_\varepsilon.
  \end{array}
\right.
\end{equation}
This problem is adjoint to the  problem
\begin{equation}\label{adjoint_probl_Phi}
\left\{\begin{array}{l}
\mathcal{L}^*_\varepsilon(\Psi) = 0 \quad \text{in} \ \Omega_\varepsilon,
\\
\Psi = 0\quad  \text{on} \ \Upsilon_\varepsilon,
\qquad
-\varepsilon \sigma_\varepsilon(\Psi)=0\quad  \text{on} \ \partial\Omega_\varepsilon\setminus \Upsilon_\varepsilon.
\end{array}\right.
\end{equation}
Due to our assumptions and the result of problem 3.1 \cite[\S 3]{GilTru} the solution to the problem \eqref{adjoint_probl_Phi}
is trivial. Therefore, employing the Fredholm theory, we can state that there is a unique weak solution to the problem~\eqref{probl}.

Our aim  is  to construct the asymptotic expansion for the solution $u_\varepsilon,$
  to derive the corresponding limit problem $(\varepsilon =0)$, and
 to prove the corresponding asymptotic estimates.

\section{Formal asymptotic expansions}\label{Sec:expansions}
\subsection{Regular part of the asymptotics}\label{regular_asymptotic}
We seek it  in the form
\begin{equation}\label{regul}
\mathfrak{U}_\varepsilon^{(i)} := w_0^{(i)} (x_i) + \sum\limits_{k=1}^{+\infty} \varepsilon^{k}
    \left(w_k^{(i)} (x_i) + u_k^{(i)} \left( x_i, \dfrac{\overline{x}_i}{\varepsilon} \right)
     \right).
\end{equation}
For each $i \in \{1,2,3\},$ the ansatz $\mathfrak{U}_\varepsilon^{(i)}$ is located inside of the thin cylinder $\Omega^{(i)}_\varepsilon$ and their  coefficients depend both on the corresponding longitudinal variable $x_i$ and so-called "quick variables" $\frac{\overline{x}_i}{\varepsilon}.$

Formally substituting $\mathfrak{U}_\varepsilon^{(i)}$  into the corresponding differential equation of problem~\eqref{probl}, into the boundary condition on the lateral surface of the thin cylinder $\Omega^{(i)}_\varepsilon,$ and collecting terms of the same powers of $\varepsilon,$ we obtain
\begin{multline}\label{rel_1}
    -
    \mathrm{div}_{\bar{\xi}_i}
    \big(
        \tilde{\mathbb{D}}^{(i)}(\bar{\xi}_i)
        \nabla_{\bar{\xi}_i} u^{(i)}_1(x_i, \bar{\xi}_i)
    \big)
    +
    \big(  v_i^{(i)}(x_i) \, w^{(i)}_0(x_i) \big)^\prime
    +
    w^{(i)}_0(x_i) \,
    \mathrm{div}_{\bar{\xi}_i}
    \big( \overline{V}^{(i)}(x_i, \bar{\xi}_i) \big)
\\
    +
    \sum\limits_{k=1}^{+\infty}\varepsilon^{k}
    \Bigg(
        -
        \mathrm{div}_{\bar{\xi}_i}
        \big(
            \tilde{\mathbb{D}}^{(i)}(\bar{\xi}_i)
            \nabla_{\bar{\xi}_i} u^{(i)}_{k+1}(x_i, \bar{\xi}_i)
        \big)
        -
        a_{ii}^{(i)}
        \Big(
            (w^{(i)}_{k-1}(x_i))^{\prime\prime}
            +
            (u^{(i)}_{k-1}(x_i, \bar{\xi}_i))^{\prime\prime}
        \Big)
\\
        +
        \Big(
            v^{(i)}(x_i) \,
            \big[ w^{(i)}_k(x_i) + u^{(i)}_{k}(x_i, \bar{\xi}_i) \big]
        \Big)^\prime
        +
        \mathrm{div}_{\bar{\xi}_i}
        \big(
            \overline{V}^{(i)}(x_i, \bar{\xi}_i) \,
            \big[ w^{(i)}_k(x_i) + u^{(i)}_{k}(x_i, \bar{\xi}_i) \big]
        \big)
    \Bigg)
    \approx
    0,
\end{multline}
and
\begin{multline}\label{rel_2}
-  \tilde{\mathbb{D}}^{(i)}(\bar{\xi}_i)\, \bar{\nu}_{\bar{\xi}_i} \cdot \nabla_{\bar{\xi}_i} u^{(i)}_1(x_i, \bar{\xi}_i)
+  w^{(i)}_0(x_i) \,  \bar{\nu}_{\bar{\xi}_i} \cdot \overline{V}^{(i)}(x_i, \bar{\xi}_i)
\\
+ \sum\limits_{k=1}^{+\infty}\varepsilon^{k}
\Big(- \tilde{\mathbb{D}}^{(i)}(\bar{\xi}_i)\, \bar{\nu}_{\bar{\xi}_i}\cdot \nabla_{\bar{\xi}_i} u^{(i)}_{k+1}(x_i, \bar{\xi}_i) +
\big[w^{(i)}_k(x_i) + u^{(i)}_{k}(x_i, \bar{\xi}_i)\big]\,  \bar{\nu}_{\bar{\xi}_i} \cdot \overline{V}^{(i)}(x_i, \bar{\xi}_i) \Big)  \approx \varphi^{(i)}(x_i),
\end{multline}
where $\bar{\xi}_i=\frac{\bar{x}_i}{\varepsilon},$ the symbol ``$\prime$'' denotes the derivative with respect to $x_i,$
$u^{(i)}_{0} \equiv 0,$ $\bar{\nu}_{\bar{\xi}_i}$ is the outward unit normal to the boundary of the 
disk  $\Upsilon_i(x_i):=\big\{ \overline{\xi}_i\in\Bbb{R}^2 \ : \ |\overline{\xi}_i|< h_i \big\},$
\begin{gather*}
\overline{V}^{(1)}(x_1, \bar{\xi}_1) =
\Big( v^{(1)}_2(x_1, \bar{\xi}_1), \, v^{(1)}_3(x_1,\bar{\xi}_1) \Big),
\qquad
\overline{V}^{(2)}(x_2, \bar{\xi}_2) =
\Big( v^{(2)}_1(x_2, \bar{\xi}_2), \, v^{(2)}_3(x_2,\bar{\xi}_2) \Big),
\\[2mm]
\overline{V}^{(3)}(x_3, \bar{\xi}_3) =
\Big( v^{(3)}_1(x_3, \bar{\xi}_3), \, v^{(3)}_2(x_3,\bar{\xi}_3) \Big),
\end{gather*}
and
\begin{gather}
\notag
\tilde{\mathbb{D}}^{(1)}(\bar{\xi}_1)
 =
\left(
    \begin{matrix}
        a^{(1)}_{22}(\bar{\xi}_1) & a^{(1)}_{23}(\bar{\xi}_1)
        \\[2mm]
        a^{(1)}_{32}(\bar{\xi}_1) & a^{(1)}_{33}(\bar{\xi}_1)
    \end{matrix}
\right),
\qquad
\tilde{\mathbb{D}}^{(2)}(\bar{\xi}_2)
 =
\left(
    \begin{matrix}
        a^{(2)}_{11}(\bar{\xi}_2) & a^{(2)}_{13}(\bar{\xi}_2)
        \\[2mm]
        a^{(2)}_{31}(\bar{\xi}_2) & a^{(2)}_{33}(\bar{\xi}_2)
    \end{matrix}
\right),
\\[2mm] \label{mat-2D}
\tilde{\mathbb{D}}^{(3)}(\bar{\xi}_3)
 =
\left(
    \begin{matrix}
        a^{(3)}_{11}(\bar{\xi}_3) & a^{(3)}_{12}(\bar{\xi}_3)
        \\[2mm]
        a^{(3)}_{21}(\bar{\xi}_3) & a^{(3)}_{22}(\bar{\xi}_3)
    \end{matrix}
\right).
\end{gather}

Equating the coefficients of the same powers of $\varepsilon$ in~\eqref{rel_1} and~\eqref{rel_2}, we deduce recurrent relations of  boundary-value problems. The  problem for  $u_1^{(i)}$ is as follows:
\begin{equation}\label{regul_probl_2}
\left\{\begin{array}{rcll}
- \mathrm{div}_{\bar{\xi}_i} \Big( \tilde{\mathbb{D}}^{(i)}(\bar{\xi}_i) \nabla_{\bar{\xi}_i} u^{(i)}_1(x_i, \bar{\xi}_i)
+ w^{(i)}_0(x_i) \, \overline{V}^{(i)}(x_i, \bar{\xi}_i)\Big)
 & = & - \, \big(v_i^{(i)}(x_i) \, w^{(i)}_0(x_i)\big)^\prime ,
 & \ \ \overline{\xi}_i\in\Upsilon_i (x_i),
\\[2mm]
-   \Big(\tilde{\mathbb{D}}^{(i)}(x_i, \bar{\xi}_i) \, \nabla_{\bar{\xi}_i} u^{(i)}_1(x_i, \bar{\xi}_i) +
  w^{(i)}_0(x_i) \,   \overline{V}^{(i)}(x_i, \bar{\xi}_i)\Big) \cdot \bar{\nu}_{\bar{\xi}_i}
 & = &  \varphi^{(i)}(x_i),
 & \ \ \overline{\xi}_i\in\partial\Upsilon_i(x_i),
\\[2mm]
\langle u_1^{(i)} (x_i,\cdot) \rangle_{\Upsilon_i (x_i)}
 & = & 0.
 &
\end{array}\right.
\end{equation}
Here the variable $x_i$ is regarded as a parameter from $ \ I_\varepsilon^{(i)} :=\{x\colon \ x_i\in  (\varepsilon \ell_0, \ell_i), \ \overline{x_i}=(0,0)\},$
$$
\langle u(x_i,\cdot) \rangle_{\Upsilon_i(x_i)} :=  \int_{\Upsilon_i(x_i)}u (x_i,\overline{\xi}_i)\, d{\overline{\xi}_i}.
$$

For each value of $i\in \{1, 2, 3\},$ the problem~\eqref{regul_probl_2} is the inhomogeneous Neumann problem  in the disk $\Upsilon_i(x_i)$ with respect to the variable ${\overline{\xi}_i}\in\Upsilon_i(x_i).$  Writing down
the necessary and sufficient condition for the solvability of problem~\eqref{regul_probl_2}, we get the following differential equation for the coefficient $w_0^{(i)}:$
\begin{equation}\label{omega_probl_2}
 - h_i^2 \, \big(v_i^{(i)}(x_i) \, w^{(i)}_0(x_i)\big)_{x_i}^\prime
 = 2 h_i \, \varphi^{(i)}(x_i),
\quad x_i\in I_\varepsilon^{(i)}.
\end{equation}
Let $w_0^{(i)}$ be a solution of~\eqref{omega_probl_2}. Then, there exists a solution to problem~\eqref{regul_probl_2} and the third relation in~\eqref{regul_probl_2} supplies its uniqueness.

For determination of the other coefficients $\{u_k^{(i)}\},$ we get the following problems:
\begin{equation}\label{regul_probl_3}
\left\{\begin{array}{rcll}
    -
    \mathrm{div}_{\bar{\xi}_i}
    \big( \tilde{\mathbb{D}}^{(i)}(\bar{\xi}_i) \nabla_{\bar{\xi}_i} u^{(i)}_{k+1} \big)
    & = &
    a_{ii}^{(i)}
    \Big( (w^{(i)}_{k-1})^{\prime\prime} + (u^{(i)}_{k-1})^{\prime\prime} \Big)
    -
    \big( v_i^{(i)}(x_i) \, \big[w^{(i)}_k + u^{(i)}_{k} \big] \big)^{\prime}
\\[2mm]
    & &
    -
    \, \mathrm{div}_{\bar{\xi}_i}
    \big(
        \overline{V}^{(i)}(x_i, \bar{\xi}_i) \,
        \big[ w^{(i)}_k  + u^{(i)}_{k} \big]
    \big),
    &
    \overline{\xi}_i\in\Upsilon_i(x_i),
 \\[2mm]
    -
    \tilde{\mathbb{D}}^{(i)}(\bar{\xi}_i) \,
    \nabla_{\bar{\xi}_i} u^{(i)}_{k+1}(x_i, \bar{\xi}_i) \cdot \bar{\nu}_{\bar{\xi}_i}
    & = &
    -
    \, \big[ w^{(i)}_k(x_i) + u^{(i)}_{k}(x_i, \bar{\xi}_i) \big] \,
    \overline{V}^{(i)}(x_i, \bar{\xi}_i) \cdot \bar{\nu}_{\bar{\xi}_i} ,
    &
    \overline{\xi}_i\in\partial\Upsilon_i(x_i),
\\[2mm]
    \langle u_{k+1}^{(i)}(x_i,\cdot) \rangle_{\Upsilon_i(x_i)}
    & = &
    0.
    &
\end{array}\right.
\end{equation}
Repeating the previous reasoning and taking into account the third relation in~\eqref{regul_probl_2}, we find
\begin{equation*}
\big(v^{(i)}(x_i) \, w^{(i)}_k(x_i)\big)^\prime =
    a_{ii}^{(i)} \, (w^{(i)}_{k-1}(x_i))^{\prime\prime},
\quad x_i\in I_\varepsilon^{(i)}, \quad i \in\{1,2,3\},
\end{equation*}
whence
\begin{equation}\label{w_k}
w^{(i)}_k(x_i) = \dfrac{a_{ii}^{(i)}}{v_i^{(i)}(x_i)}\, \big(w^{(i)}_{k-1}(x_i)\big)^\prime + \dfrac{c_{k}^{(i)}}{v^{(i)}(x_i)}, \quad k \in \Bbb N, \ \ i \in\{1,2,3\},
\end{equation}
where constants $\{c_{k}^{(i)}\}$ will be defined later.
Then the solution to problem~\eqref{regul_probl_3} is uniquely determined and problem~\eqref{regul_probl_3} can now rewritten as follows:
\begin{equation}\label{problem_u_k}
\left\{\begin{array}{rcll}
    -
    \mathrm{div}_{\bar{\xi}_i}
    \big(
        \tilde{\mathbb{D}}^{(i)}(\bar{\xi}_i)
        \nabla_{\bar{\xi}_i} u^{(i)}_{k+1}(x_i, \bar{\xi}_i)
    \big)
    & = &
    a_{ii}^{(i)} \, (u^{(i)}_{k-1}(x_i, \bar{\xi}_i))^{\prime\prime}
    -
    \big( v^{(i)}(x_i) \, u^{(i)}_{k}(x_i, \bar{\xi}_i) \big)^\prime
\\[2mm]
    & &
    -
    \, \mathrm{div}_{\bar{\xi}_i}
    \big(
        \overline{V}^{(i)}(x_i, \bar{\xi}_i) \,
        \big[ w^{(i)}_k  + u^{(i)}_{k} \big]
    \big),
    &
    \overline{\xi}_i\in\Upsilon_i(x_i),
\\[2mm]
    -
    \tilde{\mathbb{D}}^{(i)}(\bar{\xi}_i) \,
    \nabla_{\bar{\xi}_i} u^{(i)}_{k+1}(x_i, \bar{\xi}_i) \cdot \bar{\nu}_{\bar{\xi}_i}
    & = &
    -
    \, \big[ w^{(i)}_k(x_i) + u^{(i)}_{k}(x_i, \bar{\xi}_i) \big] \,  \overline{V}^{(i)}(x_i, \bar{\xi}_i) \cdot \bar{\nu}_{\bar{\xi}_i} ,
    &
    \overline{\xi}_i\in\partial\Upsilon_i(x_i),
\\[2mm]
    \langle u_{k+1}^{(i)}(x_i,\cdot) \rangle_{\Upsilon_i(x_i)}
    & = &
    0.
    &
\end{array}\right.
\end{equation}

\begin{remark}\label{r_2_1}
Since the functions $\{\varphi^{(i)}\}_{i=1}^3$ have compact supports, the coefficients $\{u_k^{(i)}\}$ vanish in the corresponding neighborhoods.

Besides, as the functions $\{v^{(i)}_i\}_{i=1}^3$ are constant in neighborhoods of the origin, the coefficients $\{w_k^{(i)}\}$ are also constant in the same intervals.
\end{remark}

\subsection{Inner part of the asymptotics}\label{subsec_Inner_part}
To uniquely define the coefficients $\{w_k^{(j)}\}$ we should launch
the inner part of the asymptotics for the solution to problem ~(\ref{probl}). For this purpose we pass to the variables $\xi=\frac{x}{\varepsilon}.$ Letting $\varepsilon$ to $0,$ we see  that the domain $\Omega_\varepsilon$ is transformed into the unbounded domain $\Xi$ that  is the union of the domain~$\Xi^{(0)}$ and three semibounded cylinders
$$
\Xi^{(j)}
 =  \{ \xi=(\xi_1,\xi_2,\xi_3)\in\Bbb R^3 \ :
    \quad  \ell_0<\xi_j<+\infty,
    \quad |\overline{\xi}_j|<h_j \},
\qquad j\in \{1,2,3\},
$$
i.e., $\Xi$ is the interior of the set $\bigcup_{j=0}^3\overline{\Xi^{(j)}}$ (see Fig.~\ref{Fig-5}).

\begin{figure}[htbp]
\begin{center}
\includegraphics[width=7cm]{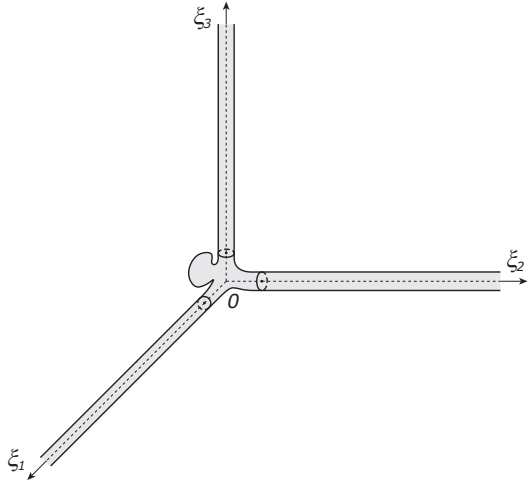}
\vskip - 10pt
\caption{Domain $\Xi$}\label{Fig-5}
\end{center}
\end{figure}

For parts of the   boundary of the domain $\Xi$ we  introduce notation
$$
\Gamma_j = \{ \xi\in\Bbb R^3 \ : \quad \ell_0 <\xi_j<+\infty, \quad |\overline{\xi}_j|=h_j\}, \quad j\in \{1,2,3\}, \quad \text{and} \quad
\Gamma_0 = \partial\Xi \backslash \left(\bigcup_{j=1}^3 \Gamma_j \right).
$$

 The following ansatz is proposed
 \begin{equation}\label{inner_part}
\mathfrak{N}_\varepsilon=\sum\limits_{k=0}^{+\infty}\varepsilon^k N_k\left(\frac{x}{\varepsilon}\right), \quad x \in \Omega^{(0)}_\varepsilon.
\end{equation}
Substituting $\mathfrak{N}_\varepsilon$  into the corresponding differential equation of problem ~(\ref{probl}),  into the boundary condition on $\partial\Omega^{(0)}_\varepsilon,$ collecting terms of the same powers of $\varepsilon$
taking into account the second part of Remark~\ref{r_2_1},  we get for each $k\in \Bbb N_0$ the following relations:

\begin{equation}\label{N_k_prob}
\left\{\begin{array}{rcll}
  -   \mathrm{div}_\xi \big( \mathbb{D}(\xi) \nabla_\xi N_k(\xi) \big) +
  \mathrm{div}_\xi \big( \overrightarrow{V}(\xi) N_k(\xi)\big)
  & = & 0, &
    \xi \in\Xi,
\\[2mm]
-\sigma_\xi(N_k(\xi)) +  N_k(\xi)\, \overrightarrow{V}(\xi) \cdot \boldsymbol{\nu}_{\xi} &=&  0, &
   \xi \in \partial\Xi,
\\[2mm]
N_k(\xi)     \  \sim \  w^{(j)}_{k}(0)        &\text{as}  &\xi_j \to +\infty, \quad \  j\in \{1,2,3\},&
  \xi  \in \Xi^{(j)},
 \\[2mm]
 N_k\big|_{\xi_j= \ell_0-0} &=& N_k\big|_{\xi_j= \ell_0+0}, &  j\in \{1,2,3\},
 \\[2mm]
 a_{jj}^{(j)} \, \partial_{\xi_j}N_k\big|_{\xi_j= \ell_0+0} &=&
 \sum_{l=1}^{3 } \big(a_{jl}^{(0)}(\xi) \, \partial_{\xi_l} N_k\big)\big|_{\xi_j= \ell_0-0}, &
 j\in \{1,2,3\},
\end{array}\right.
\end{equation}
where $\sigma_\xi(N_k) := \boldsymbol{\nu}_{\xi} \cdot {\mathbb{D}} \nabla_{\xi} N_k ,$
$\boldsymbol{\nu}_{\xi}$ is the outward unit normal to $\partial \Xi,$
 $\mathbb{D}(\xi)= \mathbb{D}^{(0)}(\xi) := \mathbb{D}_\varepsilon^{(0)}(x)\big|_{x= \varepsilon \xi}$ if $\xi \in \Xi^{(0)},$
 $$
\mathbb{D}(\xi) = \mathbb{D}^{(1)}(\overline{\xi}_1) :=
\left(
\begin{matrix}
  a^{(1)}_{11} & 0 & 0 \\[2mm]
  0 & a^{(1)}_{22}(\overline{\xi}_1)  &  a^{(1)}_{23}(\overline{\xi}_1) \\[2mm]
  0 &  a^{(1)}_{32}(\overline{\xi}_1) &  a^{(1)}_{33}(\overline{\xi}_1)
\end{matrix}
\right), \quad \xi \in \Xi^{(1)},
$$

$$
\mathbb{D}(\xi)  = \mathbb{D}^{(2)}(\overline{\xi}_2) :=
\left(
\begin{matrix}
  a^{(2)}_{11}(\overline{\xi}_2) & 0 & a^{(2)}_{13}(\overline{\xi}_2) \\[2mm]
  0 & a^{(2)}_{22} & 0 \\[2mm]
  a^{(2)}_{31}(\overline{\xi}_2) & 0 & a^{(2)}_{33}(\overline{\xi}_2)
\end{matrix}
\right),  \quad \xi \in \Xi^{(2)},
$$

$$
\mathbb{D}(\xi) = \mathbb{D}^{(3)}(\overline{\xi}_3) :=
\left(
\begin{matrix}
  a^{(3)}_{11}(\overline{\xi}_3) & a^{(3)}_{12}(\overline{\xi}_3) & 0 \\[2mm]
  a^{(3)}_{21}(\overline{\xi}_3) & a^{(3)}_{22}(\overline{\xi}_3) & 0 \\[2mm]
  0 & 0 & a^{(3)}_{33}
\end{matrix}
\right), \quad \xi \in \Xi^{(3)}
$$
(the coefficients $\{a^{(j)}_{jj}\}_{j=1}^3$ are positive constants due to our assumptions),
$\overrightarrow{V}(\xi) = \overrightarrow{V}_\varepsilon^{(0)}(x)\big|_{x= \varepsilon \xi}$  if $\xi \in \Xi^{(0)},$
$\overrightarrow{V}(\xi) = \left( \mathrm{v}_1, \, 0,\, 0\right)$ if $\xi \in \Xi^{(1)},$
$\overrightarrow{V}(\xi) = \left(0, \,  \mathrm{v}_2,\, 0\right)$ if $\xi \in \Xi^{(2)},$ and
$\overrightarrow{V}(\xi) = \left( 0,\, 0, \, \mathrm{v}_3\right)$ if $\xi \in \Xi^{(3)}.$
\
Relations in the third line of \eqref{N_k_prob} appear by matching the regular and inner asymptotics in a neighborhood of the node, namely the asymptotics of the terms $\{N_k\}$ as $\xi_j \to +\infty$ have to coincide with the corresponding asymptotics of  terms of the regular expansions (\ref{regul}) as $x_j =\varepsilon \xi_j \to +0, \ j\in \{1,2,3\},$ respectively.

Thus, we must  find a solution to the problem \eqref{N_k_prob} so  for each $j\in \{1,2,3\}$ it  stabilizes to the constant $w^{(j)}_{k}(0)$ as $\xi_j \to +\infty$  in the outlet $\Xi^{(j)}.$
We look for a solution to the problem \eqref{N_k_prob} with a fixed index  $k\in \Bbb N$  in the form
\begin{equation}\label{new-solution}
N_k(\xi) = \sum\limits_{j=1}^3w^{(j)}_{k}(0) \,\chi_{\ell_0}(\xi_j) + \widetilde{N}_k(\xi),
\end{equation}
where $ \chi_{\ell_0} \in C^{\infty}(\Bbb{R})$ is a cut-off function such that
$\ 0\leq \chi_{\ell_0} \leq1,$  $\chi_{\ell_0}(t) =0$ if $t \leq  1+\ell_0$  and
$\chi_{\ell_0}(t) =1$ if $t \geq  2+\ell_0.$
Then $\widetilde{N}_k$ has  to be  a  solution to the problem
\begin{equation}\label{tilda_N_k_prob}
\left\{\begin{array}{rcll}
  -   \mathrm{div}_\xi \big( \mathbb{D}^{(0)}(\xi) \nabla_\xi \widetilde{N}_k(\xi) \big) +
  \mathrm{div}_\xi \big( \overrightarrow{V}(\xi) \widetilde{N}_k(\xi)\big)
  & = & 0, &
    \xi \in\Xi^{(0)},
\\[2mm]
-\sigma_\xi(\widetilde{N}_k(\xi)) +  \widetilde{N}_k(\xi)\, \overrightarrow{V}(\xi) \cdot \boldsymbol{\nu}_{\xi} &=&  0, &
   \xi \in \Gamma_0,
\\[2mm]
 -   \mathrm{div}_\xi \big( \mathbb{D}^{(j)}(\xi) \nabla_\xi \widetilde{N}_k(\xi) \big) +
  \mathrm{v}_j \, \partial_{\xi_j} \widetilde{N}_k(\xi)  & = & f_j(\xi_j), &
    \xi \in\Xi^{(j)},
\\[2mm]
-\sigma_\xi(\widetilde{N}_k(\xi))  &=&  0, &
   \xi \in \Gamma_j, \quad j\in \{1,2,3\},
\end{array}\right.
\end{equation}
and has to satisfy the conditions:
\begin{equation}\label{junc_probl_general+cond}
   \widetilde{N}_k(\xi)  \rightarrow  0
   \quad \text{as} \quad \xi_j \to +\infty, \ \  \xi  \in \Xi^{(j)}, \quad j\in \{1,2,3\},
\end{equation}
where $ \partial_{\xi_j} = \frac{\partial}{\partial_{\xi_j}},$
\begin{equation}\label{F_1}
f_j(\xi_j)  =  a^{(j)}_{jj} \, w^{(j)}_k(0) \, \chi''_{\ell_0}(\xi_j) - \mathrm{v}_j  \, w^{(j)}_k(0) \, \chi'_{\ell_0}(\xi_j) .
\end{equation}

\begin{remark}
Hereafter, to simplify the formulas, we omit the conjugation conditions  on  $\{\Upsilon_j(\ell_0)\}_{ j=1}^3$.
\end{remark}

Boundary-value problems in unbounded domains are well studied (see e.g. \cite{Ber-Nir_1990, Kon-Ole_1983,Koz-Maz-Ros_97,Na-Pla,Naz96,Naz99,Ole_book_1996}).
We will use the approach proposed in \cite[\S 2.2]{Ole_book_1996} and \cite[\S 3]{Naz99}.

Let us define a weighted Sobolev space $\mathcal{H}_\beta,$ where $\beta \in \Bbb R,$ as the set of all functions from $H^1(\Xi)$
with the finite norm
$$
\|u\|_{\beta} := \bigg(\int_{\Xi} \varrho(\xi)\Big(|\nabla_\xi u|^2 + |u|^2 \Big) d\xi  \bigg)^{1/2},
$$
where $\varrho$ is a smooth function such that
$$
\varrho(\xi) = \left\{
                 \begin{array}{ll}
                   1, & \xi \in \Xi^{(0)},
\\
                   e^{\beta \xi_j} , & \xi_j \ge 2 \ell_0, \ \ \xi \in \Xi^{(j)}, \ j\in\{1, 2, 3\}.
                 \end{array}
               \right.
$$

\begin{lemma}
There exists $\vartheta_0 > 0$ such that for all $\beta \in (0, \vartheta_0)$    the problem \eqref{tilda_N_k_prob}-\eqref{junc_probl_general+cond}  has unique solution in the space $\mathcal{H}_\beta$ if and only if
the equality
\begin{equation}\label{cong_cond_g}
\sum_{j=1}^{3}\int_{\Xi^{(j)}} f_j\,  d\xi = 0
\end{equation}
is satisfied. The equality \eqref{cong_cond_g} is equivalent to
\begin{equation}\label{cong_cond}
  \sum_{j=1}^{3} \mathrm{v}_j \, h_j^2 \, w_k^{(j)}(0) = 0.
\end{equation}
 \end{lemma}

\begin{proof}
As follows from \cite[\S 2.2]{Ole_book_1996} and \cite[\S 3]{Naz99}, we should consider in each cross-section $\Upsilon^{(j)}:= \{\bar{\xi}_j\colon |\bar{\xi}_j| < h_j\}$ of the semi-cylinder $\Xi^{(j)}$  $(j \in \{1, 2, 3\})$ the following spectral problem (for definiteness, the case $j= 1$ is considered)
\begin{equation}\label{pensil_prob}
\left\{\begin{array}{rcll}
    \displaystyle {\sum_{j,k=2}^{3}\frac{\partial}{\partial \xi_j} \Big( a_{jk}^{(1)}(\xi_2, \xi_3) \frac{\partial U(\xi_2,\xi_3)}{\partial \xi_k} \Big) } & = &
 \displaystyle { i \lambda \mathrm{v}_1  U(\xi_2,\xi_3) + \lambda^2 a_{11}^{(1)} U(\xi_2,\xi_3)}, &
    (\xi_2, \xi_3) \in\Upsilon^{(1)},
\\[2mm]
\displaystyle{\sigma_{\xi_2,\xi_3}(U)}
 &=&  0, &
   (\xi_2, \xi_3) \in \partial \Upsilon^{(1)}.
\end{array}\right.
\end{equation}
Here $i$ is the imaginary unit, \
$\displaystyle \sigma_{\xi_2,\xi_3}(U) := \sum\nolimits_{j,k=2}^{3} a_{jk}^{(1)}(\xi_2, \xi_3) \frac{\partial U(\xi_2,\xi_3)}{\partial \xi_k} \, \nu_j(\xi_2,\xi_3).$

Obviously, that $\lambda =0 $ is an eigenvalue of \eqref{pensil_prob} and $U = const$ is the corresponding eigenfunction.
Let us show that there is no adjoint function. Really, if $U_1 \equiv 1$ is the eigenfunction and $U_2$ is an adjoint function, then it must be a solution to the Neumann problem
\begin{equation}\label{pensil_prob_1}
\left\{\begin{array}{rcll}
    \displaystyle {\sum_{j,k=2}^{3}\frac{\partial}{\partial \xi_j} \Big( a_{jk}^{(1)}(\xi_2, \xi_3) \frac{\partial U_2(\xi_2,\xi_3)}{\partial \xi_k} \Big) } & = &   i \mathrm{v}_1, & (\xi_2, \xi_3) \in\Upsilon^{(1)},
\\[2mm]
\displaystyle{\sigma_{\xi_2,\xi_3}(U_2)}
 &=&  0, &
   (\xi_2, \xi_3) \in \partial \Upsilon^{(1)},
\end{array}\right.
\end{equation}
but this is impossible.

Now we will prove that only one eigenvalue $\lambda = 0$ of the spectral problem \eqref{pensil_prob} lies on the real axis.
Assume that there is another eigenvalue $\mu \in \Bbb R \setminus \{0\}$ of \eqref{pensil_prob} and $\Phi$ is an eigenfunction corresponding to $\mu.$ Then it easy to verify that the function
$$
\Psi(\xi) = e^{i\mu \xi_1} \, \Phi(\xi_2, \xi_3), \quad  \xi = (\xi_1, \xi_2, \xi_3) \in \mathfrak{C}:= (-\infty, + \infty) \times \Upsilon^{(1)},
$$
is a solution to the problem
\begin{equation}\label{prob_cylinder}
\left\{\begin{array}{rcll}
    \displaystyle {- a_{11}^{(1)} \frac{\partial^2 \Psi}{\partial \xi_1^2}  - \sum_{j,k=2}^{3}\frac{\partial}{\partial \xi_j} \Big( a_{jk}^{(1)}(\xi_2, \xi_3) \frac{\partial \Psi}{\partial \xi_k} \Big)  +  \mathrm{v}_1 \frac{\partial \Psi}{\partial \xi_1}} & = & 0, & \text{in} \ \mathfrak{C},
\\[2mm]
\displaystyle{\sigma_{\xi_2,\xi_3}(\Psi)}
 &=&  0, &
    \text{on}\  \partial \mathfrak{C}.
\end{array}\right.
\end{equation}
Let $\Phi(\xi_2, \xi_3) = \phi_1(\xi_2, \xi_3) + i \phi_2(\xi_2, \xi_3),$ where $\phi_1 = \text{\rm Re}(\Phi),$ $\phi_2 = \text{\rm Im}(\Phi).$ Then the function
$$
\Psi_1(\xi) = \cos(\mu \xi_1) \, \phi_1(\xi_2, \xi_3) - \sin(\mu \xi_1) \, \phi_2(\xi_2, \xi_3), \quad \xi  \in \mathfrak{C},
$$
is also  a solution to the problem  \eqref{prob_cylinder}. But  $\Psi_1$ is periodic in $\xi_1$ and according to the maximum principle (see, e.g., \cite{GilTru}) it must be constant. This is possible if and only if $\mu = 0.$

The spectrum of the problem \eqref{pensil_prob} is discrete (see \cite{Goh_Kre_1969}). In addition, there are no eigenvalues of this problem in any strip $\{\lambda \in \Bbb C\colon \ |\text{\rm Im}(\lambda)| < K\}$ if $|\text{\rm Re}(\lambda)|$ is big enough.
Indeed, multiplying the differential equation of  the problem \eqref{pensil_prob} by $\bar{U},$ integrating by parts and taking \eqref{n1} into account, we derive
\begin{gather}
  \kappa_0 \int_{\Upsilon^{(1)}} |\nabla_{\xi_2, \xi_3}U|^2 d\xi_2 d\xi_3 \le - i \lambda \mathrm{v}_1  \int_{\Upsilon^{(1)}} |U|^2 d\xi_2 d\xi_3 -  \lambda^2 a_{11}^{(1)}  \int_{\Upsilon^{(1)}} |U|^2 d\xi_2 d\xi_3 \notag
\\
= \Big( \mathrm{v}_1 \, \text{\rm Im}(\lambda)  +
\big(|\text{\rm Im}(\lambda)|^2   - |\text{\rm Re}(\lambda)|^2 \big)
 a_{11}^{(1)} \Big) \int_{\Upsilon^{(1)}} |U|^2 d\xi_2 d\xi_3 \notag
\\
  \le \Big( |\mathrm{v}_1| \, K  +
\big(K^2   - |\text{\rm Re}(\lambda)|^2 \big)
 a_{11}^{(1)} \Big) \int_{\Upsilon^{(1)}} |U|^2 d\xi_2 d\xi_3. \label{eneq_1}
\end{gather}
The right-hand side in the inequality \eqref{eneq_1} is negative for sufficiently large $|\text{\rm Re}(\lambda)|$ (recall that the constant
coefficient $a_{11}^{(1)}$ is positive). Hence, this inequality is possible only in the case $U\equiv 0.$

Considering the facts obtained above, we conclude that there exists $\vartheta_0 > 0$ such that in the strip $\{\lambda \in \Bbb C\colon \ |\text{\rm Im}(\lambda)| < \vartheta_0\}$ there is only one eigenvalue $\lambda = 0$ of the problem \eqref{pensil_prob}
and the corresponding eigenfunction $U_1 \equiv 1$ has no adjoint functions (generalized eigenvectors).

As follows from \cite[\S 3]{Naz99}, a necessary and sufficient condition the unique solvability of the problem \eqref{tilda_N_k_prob}-\eqref{junc_probl_general+cond}  in the space $\mathcal{H}_\beta,$
where $\beta \in (0, \vartheta_0)$, is the equality \eqref{cong_cond_g}
that is equivalent to \eqref{cong_cond}.
\end{proof}

Thus, if we can choose functions $\{w_k^{(j)}\}$ so to satisfy condition \eqref{cong_cond}, then  problem (\ref{N_k_prob}) has a
unique solution that exponentially stabilizes to the constant $w_k^{(j)}(0)$ in each semi-cylinder $\Xi^{(j)}, \ j\in \{1, 2, 3\}.$

\subsubsection{Limit problem}\label{sub_limit_problem}
Relations \eqref{omega_probl_2} and \eqref{cong_cond}  form the problem
\begin{equation}\label{limit_prob}
 \left\{\begin{array}{rclr}
 -  h_j^2 \, \big(v^{(j)}(x_j) \, w^{(j)}_0(x_j)\big)_{x_j}^\prime &=&  2 h_j\,  \varphi^{(j)}(x_j), &    x_j\in I_j, \quad j\in \{1,2,3\},
 \\[3mm]
     \sum_{j=1}^{3} \mathrm{v}_j \, h_j^2 \, w_0^{(j)}(0) &= &0
    \\[2mm]
    w_0^{(j)}(\ell_j) & = & q_j, &  \quad j\in \{1,2\},
  \end{array}\right.
\end{equation}
where $I_j:=\{x: \ x_j\in (0,\ell_j), \ \overline{x_j}=(0,0)\},$ which is called \textit{limit problem} for problem~\eqref{probl}. The last two equalities in \eqref{limit_prob} appeared thanks to the boundary conditions on $\Upsilon^{(1)}_\varepsilon(\ell_1)$ and $\Upsilon^{(2)}_\varepsilon(\ell_2)$ in~\eqref{probl}.

From the differential equations in problem~\eqref{limit_prob} we find
\begin{equation}\label{solution}
w^{(j)}_0(x_j) =
    -
    \frac{2}{h_j \, v_j^{(j)}(x_j)}
    \Big( \int_{0}^{x_j} \varphi^{(j)}(t) \, dt  + C_j\Big),
\quad x_j\in I_j, \quad j\in \{1,2,3\}.
\end{equation}
To satisfy the boundary conditions in~\eqref{limit_prob}, we take
$$
C_j = - \frac{h_j \, v_j^{(j)}(\ell_j)}{2}\, q_j - \int_{0}^{\ell_j} \varphi^{(j)} \, dx_j, \quad j\in \{1,2\};
$$
and to satisfy the Kirchhoff condition in~\eqref{limit_prob}, we should choose
$$
C_3= - \frac{h_1 C_1 + h_2 C_2}{h_3}.
$$

Thus, the limit problem has the unique solution.
Next we can uniquely  determine the first term $N_0$ of the inner asymptotic expansion~\eqref{inner_part} as the corresponding solution to problem \eqref{N_k_prob} with $k=0.$   For each $j\in \{1,2,3\}$ the coefficient $u_1^{(j)}$ is a unique solution to problem \eqref{regul_probl_2} that now can be rewritten
in the form
\begin{equation}\label{regul_probl_2-new}
\left\{\begin{array}{rcll}
- \mathrm{div}_{\bar{\xi}_j} \big( \tilde{\mathbb{D}}^{(j)}(\bar{\xi}_j) \nabla_{\bar{\xi}_j} u^{(j)}_1(x_j, \bar{\xi}_j)\big)
 & = & \frac{2}{h_j} \,  \varphi^{(j)}(x_j),
 & \ \ \overline{\xi}_j\in\Upsilon_j (x_j),
\\[2mm]
- \bar{\nu}_{\bar{\xi}_j}\cdot \tilde{\mathbb{D}}^{(j)}(\bar{\xi}_j) \nabla_{\bar{\xi}_j} u^{(j)}_1(x_j, \bar{\xi}_j)
 & = & \varphi^{(j)}(x_j),
 & \ \ \overline{\xi}_j\in\partial\Upsilon_j(x_j),
\\[2mm]
\langle u_1^{(j)} (x_j,\cdot) \rangle_{\Upsilon_j (x_j)}
 & = & 0.
 &
\end{array}\right.
\end{equation}
Then, sequentially, we can uniquely determine all coefficients $\{u_k^{(j)}\}$ as solutions to the corresponding problems~\eqref{problem_u_k}.

In addition, we can determine all coefficients $\{w_k^{(j)}\}$ with the help of formula \eqref{w_k}. In order to unambiguously find the corresponding constants $\{c_k^{(j)}\}$, we put the ansatz \eqref{regul} with $ j\in\{1,2\}$ in the boundary conditions on
on $\Upsilon^{(1)}_\varepsilon(\ell_1)$ and $\Upsilon^{(2)}_\varepsilon(\ell_2).$ Appealing to the first part of Remark~\ref{r_2_1}, we get
$w_k^{(1)}(\ell_1)=0$ and $w_k^{(2)}(\ell_2)=0$ for $k\in \Bbb N.$ From these relations and \eqref{w_k} we find that
$$
c_{k}^{(j)} =  - a_{jj}^{(j)} \, \big(w^{(j)}_{k-1}\big)^\prime(\ell_j),
\quad k \in \Bbb N, \ \ j\in\{1,2\}.
$$

We choose the remaining constants $\{c_k^{(3)}\}_{k\in \Bbb N}$ in such a way that the solvability conditions \eqref{cong_cond} are satisfied. Now appealing to the second part of Remark~\ref{r_2_1}, we obtain from \eqref{w_k} that
$$
w^{(j)}_k(0) = \dfrac{c_{k}^{(j)}}{\mathrm{v}_j},
\quad k \in \Bbb N, \ \ j \in\{1,2,3\}.
$$
Then, from \eqref{cong_cond} we deduce that
$$
c_k^{(3)} = - \frac{h_1^2 c_k^{(1)} + h_2^2 c_k^{(2)}}{h^2_3}, \quad k\in \Bbb N.
$$

Since for each $k\in \Bbb N$  the solvability conditions \eqref{cong_cond} is fulfilled, there exists a unique solution to problem \eqref{N_k_prob}. Thus, we have uniquely determined all coefficients of the inner expansion~\eqref{inner_part}.

It remains to satisfy the boundary condition on the base $\Upsilon^{(3)}_\varepsilon(\ell_3)$ of the thin cylinder $\Omega^{(3)}_\varepsilon.$ For this we should launch the boundary-layer asymptotics.

\subsection{Boundary-layer part of the asymptotics}\label{boundary-layer}
We seek it in the form
\begin{equation}\label{prim+}
\sum\limits_{k=0}^{+\infty}\varepsilon^{k} \, \Pi_k^{(3)}\left(\frac{{x}_1}{\varepsilon}, \frac{{x}_2}{\varepsilon}, \frac{\ell_3-x_3}{\varepsilon}\right)
\end{equation}
in a neighborhood of $\Upsilon_{\varepsilon}^{(3)} (\ell_3).$

We additionally assume that component $v_3^{(3)}$ of the vector-function $\overrightarrow{V_\varepsilon}^{(3)}$ are independent of the variable $x_3$ in a neighborhood of $\Upsilon_{\varepsilon}^{(3)}(\ell_3),$ i.e.,
$$
\overrightarrow{V_\varepsilon}^{(3)} = \big( 0, 0, v_3^{(3)}(\ell_3) \big)
$$
in a neighborhood of $\Upsilon_{\varepsilon}^{(3)} (\ell_3).$ This is a technical assumption. In the general case, the coefficients need to be expanded in the  Taylor series in a neighborhood of the point $x_3=\ell_3.$

Substituting the ansatz (\ref{prim+}) into (\ref{probl})  and collecting coefficients with the same powers of $\varepsilon$, we get the following mixed boundary-value problems:
\begin{equation}\label{prim+probl}
 \left\{\begin{array}{rcll}
    - \mathrm{div}_{\bar{\xi}_3} \big( \tilde{\mathbb{D}}^{(3)}(\bar{\xi}_3) \nabla_{\bar{\xi}_3}\Pi_k^{(3)}\big)
  - a_{33}^{(3)} \, \partial^2_{\xi_3 \xi_3}\Pi_k^{(3)} - v_3^{(3)}(\ell_3) \partial_{\xi_3}\Pi_k^{(3)}
  & =    & 0,
   & \xi\in \mathfrak{C}_+^{(3)},
   \\[2mm]
  - \, \partial_{\nu_{\overline{\xi}_3}} \Pi_k^{(3)}(\xi) & =
   & 0,
   & \xi\in \mathfrak{C}_+^{(3)},
   \\[2mm]
  \Pi_k^{(3)}(\overline{\xi}_3,0) & =
   & \Phi_k,
   & \overline{\xi}_3\in\Upsilon_3(\ell_3),
   \\[2mm]
  \Pi_k^{(3)}(\xi) & \to
   & 0,
   & \xi_3\to+\infty,
 \end{array}\right.
\end{equation}\\[3mm]
where the matrix $\tilde{\mathbb{D}}^{(3)}$ is determined in~\eqref{mat-2D},  \
$\xi=(\xi_1, \xi_2, \xi_3),$ $\xi_3 = \frac{\ell_3-x_3}{\varepsilon},$ $\overline{\xi}_3 =(\xi_1, \xi_2) = \frac{\overline{x}_3}{\varepsilon},$
$$
\mathfrak{C}_+^{(3)}:=\big\{\xi: \quad \overline{\xi}_3\in\Upsilon_3(\ell_3), \quad \xi_3\in(0,+\infty)\big\},
$$
$$
\Phi_0 = q_{3} - w_{0}^{(3)}(\ell_3), \qquad
\Phi_k =  - \, w_{k}^{(3)}(\ell_3),  \quad k\in \Bbb N.
$$

For each  index $k \in \Bbb N_0,$ with the help of the  the Fourier method, we determine the solution
\begin{equation}\label{view_solution}
\Pi_k^{(3)}(\xi) =
\Phi_k \, \exp \bigg( - \frac{v_3^{(3)}(\ell_3)}{a_{33}^{(3)}} \, \xi_3 \bigg)
\end{equation}
of problem (\ref{prim+probl}). Since  $v_3^{(3)}(\ell_3) > 0$ and $a_{33}^{(3)} >0,$
the solution $\Pi_k^{(3)}$ tends to zero as $\xi_3\to+\infty.$

\section{Complete asymptotic expansion and its justification}\label{Sec:justification}

Thus,  we  can successively determine all coefficients of series (\ref{regul}), (\ref{inner_part}) and (\ref{prim+}).
With their help, we construct the following series:

\begin{equation}\label{asymp_expansion}
    \mathcal{R}_\varepsilon =
        \sum\limits_{k = 0}^{+\infty} \varepsilon^k
        \Big(
            \overline{  u}_k (x, \,\varepsilon)
          + \overline{  N}_k (x, \,\varepsilon)
          + \overline{\Pi}_k (x, \,\varepsilon)
        \Big), \quad
    x\in\Omega_\varepsilon.
\end{equation}
where
\begin{gather*}
    \overline{u}_k (x, \,\varepsilon)
     := \sum\limits_{i=1}^3
        \chi_{\delta, 0}^{(i)} (x_i)
        \Big(
            w_k^{(i)} (x_i)
          + u_k^{(i)} \Big( x_i, \frac{\overline{x}_i}{\varepsilon} \Big)
        \Big), \quad
   (u_0 \equiv 0),
\\
    \overline{N}_k (x,\, \varepsilon)
     := \Bigg(
            1
          - \sum\limits_{i = 1}^3
            \chi_{\delta, 0}^{(i)} (x_i)
        \Bigg)
        N_k \Big( \frac{x}{\varepsilon} \Big),
\\
    \overline{\Pi}_k (x, \varepsilon)
     := \chi_\delta^{(3)} (x_3) \,
        \Pi_k^{(3)}
        \Big(
            \frac{\overline{x}_3}{\varepsilon},
            \frac{\ell_3 - x_3}{\varepsilon}
        \Big), \quad
        k\in\mathbb{N}_0,
\end{gather*}
$\chi_{\delta, 0}^{(i)}, \ \chi_\delta^{(3)}$
are smooth cut-off functions defined by formulas
\begin{equation}\label{cut-off_functions}
\chi_{\delta, 0}^{(i)} (x_i)
  = \left\{
    \begin{array}{ll}
        1, & \text{if} \ \ x_i \ge \, \varepsilon \ell_0 + 2 \delta,
    \\
        0, & \text{if} \ \ x_i \le \, \varepsilon \ell_0 + \delta,
    \end{array}
    \right.
\quad
\chi_\delta^{(3)} (x_3) =
    \left\{
    \begin{array}{ll}
        1, & \text{if} \ \ x_3 \ge \ell_3 -  \delta,
    \\
        0, & \text{if} \ \ x_3 \le \ell_3 - 2\delta,
    \end{array}
    \right.
\quad i = 1, 2, 3,
\end{equation}
and $\delta$ is a sufficiently small fixed positive number.

\bigskip

To justify the asymptotics constructed above we additionally assume that
\begin{equation}\label{assum_1}
  \forall\, i\in \{1, 2, 3\} \quad \forall\, x_i \in (0, \ell_i) \colon \ \ \frac{\partial v_i^{(i)}(x_i)}{\partial x_i} \ge 0;
\end{equation}
\begin{equation}\label{assum_2}
  \forall\, i\in \{1, 2, 3\} \colon
  \quad \kappa_0^{(i)} - d_i \ell_i > 0,
\end{equation}
where
$d_1 := \sup_{{\Omega^{(1)}_\varepsilon}} |(v_2^{(1)}, v_3^{(1)})|$
$d_2 := \sup_{{\Omega^{(2)}_\varepsilon}} |(v_1^{(2)}, v_3^{(2)})|$
$d_3 := \sup_{{\Omega^{(3)}_\varepsilon}} |(v_1^{(3)}, v_2^{(3)})|$;
\begin{equation}\label{cond_1}
  \sum_{i=1}^{3} h_i^2 \, \mathrm{v}_i = 0;
\end{equation}
and  the velocity field is conservative  in the node $\Omega_\varepsilon^{(0)}$   and
its potential $p$ is a solution of the boundary-value problem for the Laplacian
\begin{equation}\label{potential}
\left\{\begin{array}{rcll}
    \Delta_\xi p(\xi)  & = & 0, &
    \xi \in \Xi^{(0)},
\\[2mm]
\frac{\partial p(\xi)}{\partial \xi_i}  &=&  \mathrm{v}_i, &
   \xi \in \Upsilon^{(i)}(\ell_0) := \overline{\Xi^{(0)}} \cap
\big\{ \xi\colon \xi_i= \ell_0\big\},
\\[2mm]
\frac{\partial p(\xi)}{\partial \boldsymbol{\nu}}  &=&  0, &
   \xi \in \Gamma^{(0)} := \partial{\Xi^{(0)}} \setminus \Big( \bigcup_{i=1}^3 \Upsilon^{(i)}(\ell_0)\Big),
\end{array}\right.
\end{equation}
where $\frac{\partial p}{\partial \boldsymbol{\nu}}$ is the derivative along the  outward unit normal $\boldsymbol{\nu}$ to $\partial \Xi^{(0)}.$

Obviously, that due to \eqref{cond_1} the Neumann problem \eqref{potential} has a weak solution. It  is not unique and for its uniqueness  we will regard that $\int_{\Xi^{(0)}} p(\xi)\, d\xi =0.$ Thus,
\begin{equation}\label{field_0}
  \overrightarrow{V_\varepsilon}^{(0)}(x)  = \nabla_{\xi}p(\xi)\Big|_{\xi = \frac{x}{\varepsilon}} = \varepsilon \nabla_x \big( p(\frac{x}{\varepsilon}) \big), \quad x\in \Omega^{(0)}_\varepsilon.
\end{equation}

It is well known that the Laplacian of a velocity potential is equal to the divergence of the corresponding flow. Hence, $\overrightarrow{V_\varepsilon}^{(0)}$  is incompressible, i.e.,
 $\mathrm{div}\overrightarrow{V_\varepsilon}^{(0)} = 0$ in $\Omega^{(0)}_\varepsilon.$

\begin{theorem}\label{apriory_estimate}
Let the velocity field $\overrightarrow{V_\varepsilon}^{(0)}$ be  conservative  in $\Omega_\varepsilon^{(0)},$
relations \eqref{assum_1} --  \eqref{cond_1} hold, and  functions $F_i \in L^2(\Omega_\varepsilon^{(i)}), \ \Phi_i \in L^2(\Gamma_\varepsilon^{(i)}), \  i\in\{1,2,3\}$. Then for a weak solution to the problem
   \begin{equation}\label{probl+}
\left\{\begin{array}{rcll}
  -  \varepsilon\, \mathrm{div} \big( \mathbb{D}^{(i)}_\varepsilon \nabla \psi_\varepsilon\big) +
  \mathrm{div} \big( \overrightarrow{V_\varepsilon}^{(i)} \, \psi_\varepsilon\big)
  & = & F_i, &
    \text{in} \ \Omega_\varepsilon^{(i)}, \ \   i\in\{1,2,3\},
\\[2mm]
   -  \varepsilon \, \sigma_\varepsilon(\psi_\varepsilon)  + \psi_\varepsilon\,  \overrightarrow{V_\varepsilon}^{(i)} \cdot \boldsymbol{\nu}
  & = & \varepsilon\, \Phi_i &
  \text{on} \  \Gamma_\varepsilon^{(i)}, \ \   i\in\{1,2,3\},
\\[2mm]
-  \varepsilon\, \mathrm{div} \big( \mathbb{D}^{(0)}_\varepsilon \nabla \psi_\varepsilon\big) +
  \mathrm{div} \big(\overrightarrow{V_\varepsilon}^{(0)} \, \psi_\varepsilon\big)  & = & 0, &
    \text{in} \ \Omega_\varepsilon^{(0)},
\\[2mm]
 -  \varepsilon \, \sigma_\varepsilon(\psi_\varepsilon) +  \psi_\varepsilon \, \overrightarrow{V_\varepsilon}^{(0)} \cdot \boldsymbol{\nu}  & = & 0  &
    \text{on} \ \Gamma_\varepsilon^{(0)},
\\[2mm]
 \psi_\varepsilon \big|_{x_i= \ell_i}
 & = & 0, & \text{on} \ \Upsilon_{\varepsilon}^{(i)} (\ell_i) , \ \  i\in\{1,2,3\},
 \\[2mm]
 \psi_\varepsilon\big|_{x_i= \varepsilon \ell_0-0} &=& \psi_\varepsilon\big|_{x_i= \varepsilon \ell_0+0}, &  i\in\{1,2,3\},
 \\[2mm]
 a_{ii}^{(i)} \, \partial_{x_i}\psi_\varepsilon\big|_{x_i= \varepsilon \ell_0+0} &=&
 \sum_{j=1}^{3 } \big(a_{ij}^{(0)} \partial_{x_j} \psi_\varepsilon\big)\big|_{x_i= \varepsilon \ell_0-0}, &
 i\in\{1,2,3\},
\end{array}\right.
\end{equation}
 we have the estimate
\begin{equation}\label{appriory_estimate}
  \|\psi_\varepsilon\|_{H^1(\Omega_\varepsilon)} \le C \varepsilon^{-1} \sum_{i=1}^{3}\Big( \|F_i\|_{L^2(\Omega_\varepsilon^{(i)})} + \sqrt{\varepsilon} \, \|\Phi_i\|_{L^2(\Gamma_\varepsilon^{(i)})} \Big),
\end{equation}
where the constant $C$ does not depend on  $\varepsilon, \, \psi_\varepsilon, \, \{F_i\}, \, \{\Psi_i\}.$
\end{theorem}

\begin{proof}
 Multiplying the differential equations of the  problem \eqref{probl+} by  $\psi_\varepsilon$ and integrating by parts, we get
 \begin{equation*}
  \varepsilon\int \limits_{\Omega_\varepsilon}
    (\mathbb{D}_\varepsilon {\nabla \psi_\varepsilon}) \cdot {\nabla \psi_\varepsilon} \, dx
  - \int \limits_{\Omega_\varepsilon} \psi_\varepsilon\,
    \overrightarrow{V_\varepsilon} \cdot
    {\nabla \psi_\varepsilon} \, dx
  = \sum\limits_{i=1}^{3} \
    \int \limits_{\Omega_\varepsilon^{(i)}} F_i \, \psi_\varepsilon \, dx
 -  \varepsilon \sum \limits_{i=1}^{3} \
    \int \limits_{\Gamma_\varepsilon^{(i)}} \Phi_i \, \psi_\varepsilon \, dS_x.
   \end{equation*}
In view of the ellipticity conditions \eqref{n1} and the Cauchy-Bunyakovsky-Schwarz inequality, we have
\begin{equation*}
\varepsilon \sum_{i=0}^{3} \kappa_0^{(i)} \| {\nabla \psi_\varepsilon}\|^2_{L^2(\Omega^{(i)}_\varepsilon)}
  - \int \limits_{\Omega_\varepsilon} \psi_\varepsilon\,
    \overrightarrow{V_\varepsilon} \cdot
    {\nabla \psi_\varepsilon} \, dx \le
   \sum_{i=1}^{3} \|F_i\|_{L^2(\Omega_\varepsilon^{(i)})} \|\psi_\varepsilon \|_{L^2(\Omega_\varepsilon^{(i)})}
  + \varepsilon
  \sum_{i=1}^{3} \|\Phi_i\|_{L^2(\Gamma_\varepsilon^{(i)})}  \|\psi_\varepsilon\|_{L^2(\Gamma_\varepsilon^{(i)})}.
   \end{equation*}
Due to the uniform Dirichlet conditions on the bases $\{\Upsilon_{\varepsilon}^{(i)}(\ell_i)\}_{i=1}^3$ we deduce that
\begin{equation}\label{p_0}
\|\psi_\varepsilon \|_{L^2(\Omega_\varepsilon^{(i)})} \le \ell_i \|\nabla \psi_\varepsilon \|_{L^2(\Omega_\varepsilon^{(i)})},
\end{equation}
and with the help of the inequality
$$
 \varepsilon \int_{\Gamma_\varepsilon^{(i)}} v^2 \, d\sigma_x
 \leq C_2 \Bigg( \varepsilon^2 \int_{\Omega_\varepsilon^{(i)}} |\nabla_{x}v|^2 \, dx
 +    \int_{\Omega_\varepsilon^{(i)}} v^2 \, dx \Bigg), \ \ (i\in \{1, 2, 3\})
$$
 proved in \cite{M-MMAS-2008}, we obtain
\begin{equation}\label{p_1}
\varepsilon \sum_{i=0}^{3} \kappa_0^{(i)} \| {\nabla \psi_\varepsilon}\|^2_{L^2(\Omega^{(i)}_\varepsilon)}
  - \int \limits_{\Omega_\varepsilon} \psi_\varepsilon\,
    \overrightarrow{V_\varepsilon} \cdot
    {\nabla \psi_\varepsilon} \, dx \le
    C \sum_{i=1}^{3}\Big( \|F_i\|_{L^2(\Omega_\varepsilon^{(i)})} + \sqrt{\varepsilon}\|\Phi_i\|_{L^2(\Gamma_\varepsilon^{(i)})} \Big)  \| {\nabla \psi_\varepsilon}\|_{L^2(\Omega_\varepsilon)}.
\end{equation}

 Now let's tackle the integral in the left-hand side of \eqref{p_1}. It is equal to the sum of four integrals over  the respective parts of the domain $\Omega_\varepsilon.$ Using the representation \eqref{field_0},
 \begin{multline}\label{p_2}
  - \int_{\Omega^{(0)}_\varepsilon} \psi_\varepsilon\,
    \overrightarrow{V_\varepsilon}^{(0)} \cdot
    {\nabla \psi_\varepsilon} \, dx = - \frac{\varepsilon}{2} \int_{\Omega^{(0)}_\varepsilon}     \nabla_x \big( p(\frac{x}{\varepsilon}) \big)  \cdot {\nabla\big(\psi^2_\varepsilon}\big) \, dx
    \\
    =
     - \frac{\varepsilon}{2} \int_{\Omega^{(0)}_\varepsilon}  \mathrm{div}\Big(\psi^2_\varepsilon\,  \nabla_x \big( p(\frac{x}{\varepsilon}) \big)\Big)  dx = - \sum_{i=1}^{3}  \frac{\mathrm{v}_i}{2}
     \int_{\Upsilon_{\varepsilon}^{(i)} (\varepsilon\ell_0)}\psi^2_\varepsilon\, d\bar{x}_i.
 \end{multline}
 Now, for definiteness, consider the integral over the first cylinder and make the following calculations:
 \begin{multline}\label{p_3}
  - \int_{\Omega^{(1)}_\varepsilon} \psi_\varepsilon\,
    \overrightarrow{V_\varepsilon}^{(1)} \cdot
    {\nabla \psi_\varepsilon} \, dx = - \frac{1}{2} \int_{\Omega^{(1)}_\varepsilon} v_1^{(1)}(x_1) \partial_{x_1}(\psi^2_\varepsilon) \, dx
    - \varepsilon \int_{\Omega^{(1)}_\varepsilon} \psi_\varepsilon \big(
    v^{(1)}_2(x_1, \tfrac{\overline{x}_1}{\varepsilon}),  v^{(1)}_3(x_1,\tfrac{\overline{x}_1}{\varepsilon})\big) \cdot \nabla_{\overline{x}_1}
    \psi_\varepsilon\, dx
    \\
    = \frac{\mathrm{v}_1}{2}
     \int_{\Upsilon_{\varepsilon}^{(1)} (\varepsilon\ell_0)}\psi^2_\varepsilon\, d\bar{x}_1 +
     \frac{1}{2}
     \int_{\Omega^{(1)}_\varepsilon} \psi^2_\varepsilon\, \partial_{x_1}(v_1^{(1)}(x_1))\,  dx - \varepsilon \int_{\Omega^{(1)}_\varepsilon} \psi_\varepsilon \big( v^{(1)}_2(x_1, \tfrac{\overline{x}_1}{\varepsilon}),  v^{(1)}_3(x_1,\tfrac{\overline{x}_1}{\varepsilon})\big) \cdot \nabla_{\overline{x}_1}
    \psi_\varepsilon\, dx
  \end{multline}
Summing all these integral and taking \eqref{assum_1} into account, the inequality \eqref{p_1} can be rewritten as follows
 \begin{equation*}
\varepsilon \sum_{i=0}^{3} \kappa_0^{(i)} \| {\nabla \psi_\varepsilon}\|^2_{L^2(\Omega^{(i)}_\varepsilon)}
  - \varepsilon \sum_{i=1}^{3} d_i \int_{\Omega^{(i)}_\varepsilon} |\psi_\varepsilon| \,
    |\nabla_{\bar{x}_i} \psi_\varepsilon|\, dx \le
    C \sum_{i=1}^{3}\Big( \|F_i\|_{L^2(\Omega_\varepsilon^{(i)})} + \sqrt{\varepsilon}\|\Phi_i\|_{L^2(\Gamma_\varepsilon^{(i)})} \Big)  \| {\nabla \psi_\varepsilon}\|_{L^2(\Omega_\varepsilon)}.
\end{equation*}
whence, using the inequalities \eqref{p_0}, we obtain
 \begin{equation}\label{p_4}
 \kappa_0^{(0)} \| {\nabla \psi_\varepsilon}\|_{L^2(\Omega^{(0)}_\varepsilon)} +
  \sum_{i=1}^{3} (\kappa_0^{(i)} - d_i \ell_i) \| {\nabla \psi_\varepsilon}\|_{L^2(\Omega^{(i)}_\varepsilon)}
   \le
    C \varepsilon^{-1} \sum_{i=1}^{3}\Big( \|F_i\|_{L^2(\Omega_\varepsilon^{(i)})} + \sqrt{\varepsilon}\|\Phi_i\|_{L^2(\Gamma_\varepsilon^{(i)})} \Big).
\end{equation}

 To complete the proof, we need to recall the assumption \eqref{assum_2} and the inequality
\begin{equation}\label{est4}
\int_{\Omega_\varepsilon} v^2 \, dx \le C  \int_{\Omega_\varepsilon} |\nabla_{x}v|^2 \, dx \quad \forall \, v \in
H^1(\Omega_\varepsilon), \ \ v\Big|_{\{\Upsilon_{\varepsilon}^{(i)}(\ell_i)\}_{i=1}^3} = 0,
\end{equation}
which was proved in \cite{M-AA-2021}.
\end{proof}

\begin{remark}
Hereinafter, all constants in inequalities are independent of the parameter~$\varepsilon.$
\end{remark}

\begin{theorem}\label{mainTheorem}
Series $(\ref{asymp_expansion})$ is the asymptotic expansion for the solution to the  problem $(\ref{probl})$ in the Sobolev space $H^1(\Omega_\varepsilon),$ i.e.,
\begin{equation}\label{main_estimate}
    \forall \, m \in\Bbb{N} \ \
    \exists \, {C}_m >0 \ \
    \exists \, \varepsilon_0>0 \ \
    \forall\, \varepsilon\in(0, \varepsilon_0) : \qquad
    \| \, u_\varepsilon - U_\varepsilon^{(m)} \|_{H^1(\Omega_\varepsilon)}
    \leq
    {C}_m \ \varepsilon^{m},
\end{equation}
where
\begin{equation}\label{aaN}
U^{(m)}_{\varepsilon} (x)
  = \sum\limits_{k=0}^{m} \varepsilon^{k}
    \Big(
        \overline{u  }_k (x, \,\varepsilon)
      + \overline{N  }_k (x, \,\varepsilon)
      + \overline{\Pi}_k (x, \,\varepsilon)
    \Big),
\quad x\in\Omega_\varepsilon,
\end{equation}
is the partial sum of $(\ref{asymp_expansion}).$
\end{theorem}

\begin{proof}
Take an arbitrary $m\in\mathbb{N}$.
Substituting the partial sum~\eqref{aaN} in the equations and the boundary conditions of problem~\eqref{probl} and taking into account relations \eqref{regul_probl_2}, \eqref{regul_probl_3}, \eqref{N_k_prob}, \eqref{prim+probl} for the coefficients of series~\eqref{asymp_expansion}, we find
\begin{equation*}
\begin{array}{rl}
    \varepsilon\, \mathrm{div}
    \big( \mathbb{D}^{(i)}_\varepsilon \nabla U_\varepsilon^{(m)} \big)
  - \mathrm{div}
    \big( \overrightarrow{V_\varepsilon}^{(i)} U_\varepsilon^{(m)} \big)
    \ = \
    R^{(m)}_{\varepsilon, j, (i)}
 =: R^{(m)}_{\varepsilon, (i)} &
    \text{in} \ \Omega_\varepsilon^{(i)}, \ \ i \in \{1, 2, 3\},
\\[2mm]
    \varepsilon \, \sigma_\varepsilon(U_\varepsilon^{(m)})
  - \overrightarrow{V_\varepsilon}^{(i)} \cdot \boldsymbol{\nu} \,
    U_\varepsilon^{(m)}
  + \varepsilon\, \varphi^{(i)}
    \ = \
    \breve{R}_{\varepsilon, 4, (i)}^{(m)}
 =: \breve{R}_{\varepsilon, (i)}^{(m)} &
    \text{on} \ \Gamma_\varepsilon^{(i)}, \ \ i \in \{1, 2, 3\},
\\[2mm]
    \varepsilon\, \mathrm{div}
    \big( \mathbb{D}^{(0)}_\varepsilon \nabla U_\varepsilon^{(m)} \big)
  - \mathrm{div}
    \big( \overrightarrow{V_\varepsilon}^{(0)} U_\varepsilon^{(m)} \big)
    \ = \ 0 &
    \text{in} \ \Omega_\varepsilon^{(0)},
\\[2mm]
    \varepsilon \, \sigma_\varepsilon(U_\varepsilon^{(m)})
  - \overrightarrow{V_\varepsilon}^{(0)} \cdot \boldsymbol{\nu} \,
    U_\varepsilon^{(m)}
    \ = \ 0 &
    \text{on} \ \Gamma_\varepsilon^{(0)},
\\[2mm]
  - U_\varepsilon^{(m)} \big|_{x_i = \ell_i}
  + q_i
    \ = \ 0 &
    \text{on} \ \Upsilon_{\varepsilon}^{(i)} (\ell_i) , \ \  i \in \{1, 2, 3\},
\\[2mm]
  - U_\varepsilon^{(m)} \big|_{x_i = \varepsilon \ell_0-0}
  + U_\varepsilon^{(m)} \big|_{x_i = \varepsilon \ell_0+0}
    \ = \ 0, & i \in \{1, 2, 3\},
\\[2mm]
  - a_{ii}^{(i)} \,
    \partial_{x_i} U_\varepsilon^{(m)} \big|_{x_i = \varepsilon \ell_0+0}
  + \sum_{j=1}^{3}
    \big(a_{ij}^{(0)}(\frac{x}{\varepsilon})
        \partial_{x_j} U_\varepsilon^{(m)}
    \big)
    \big|_{x_i = \varepsilon \ell_0-0}
    \ = \ 0, & i = 1, 2, 3,
\end{array}
\end{equation*}
where
\begin{multline}\label{R1}
R^{(m)}_{\varepsilon, 1, (i)}(x)
  = \varepsilon^{m} \
    \chi_{\delta, 0}^{(i)} (x_i)
    \Bigg(
        a_{ii}^{(i)} \,
        \dfrac{\partial^{\,2} u_{m-1}^{(i)}}{\partial^{\,2} {x}_i}
        \big(x_i, \bar{\xi}_i\big)
      - \dfrac{\partial}{\partial{x}_i}
        \Big(
            v_{i}^{(i)}(x_i)\, u_{m}^{(i)}\big(x_i, \bar{\xi}_i\big)
        \Big)
\\
      - \mathrm{div}_{\bar{\xi}_i}
        \bigg(
            \overline{V}^{(i)}(x_i, \bar{\xi}_i) \,
            \Big(
                w^{(i)}_{m}(x_i) + u^{(i)}_{m}(x_i, \bar{\xi}_i)
            \Big)
        \bigg)
      + \varepsilon \, a_{ii}^{(i)} \,
        \dfrac{\partial^{\,2}}{\partial^{\,2}{x}_i}
        \Big(
            w^{(i)}_{m}(x_i) + u^{(i)}_{m}\big(x_i, \bar{\xi}_i\big)
        \Big)
    \Bigg)
    \Bigg|_{\bar{\xi}_i = \frac{\overline{x}_i}{\varepsilon}},
\end{multline}
\begin{multline}\label{R2}
R^{(m)}_{\varepsilon, 2, (i)}(x)
  = \sum\limits_{k=0}^{m}\varepsilon^k
    \Bigg(
        \frac{d\chi_{\delta, 0}^{(i)}}{dx_i} (x_i)
        \bigg(
            2 \, a_{ii}^{(i)} \frac{\partial{N}_{k}}{\partial{\xi_i}}(\xi)
          - \mathrm{v}_i \, \big( N_{k}(\xi) - w_k^{(i)}(0) \big)
        \bigg)
\\
      - \varepsilon
       \frac{d^2\chi_{\delta, 0}^{(i)}}{dx_i^2} (x_i) \,
        a_{ii}^{(i)} \big( N_{k}(\xi) - w_k^{(i)}(0) \big)
    \Bigg)
    \Bigg|_{
        \zeta_i = \frac{x_i}{\varepsilon^{\alpha}}, \,
        \xi = \frac{x}{\varepsilon}
    },
\end{multline}
\begin{multline}\label{R3}
R^{(m)}_{\varepsilon, 3, (1)} \equiv R^{(m)}_{\varepsilon, 3, (2)} \equiv 0,
\quad
R^{(m)}_{\varepsilon, 3, (3)}(x)
  = \sum\limits_{k=0}^{m} \varepsilon^{k}
    \Bigg(
      - \frac{d\chi_\delta^{(3)}}{dx_i} (x_i) \,
        \bigg(
            2 \, a_{33}^{(3)} \,
            \frac{\partial\Pi_{k}^{(3)}}{\partial{\xi}_3} (\xi)
          + v_3^{(3)}(\ell_3) \, \Pi_{k}^{(3)} (\xi)
        \bigg)
\\
      + \varepsilon \frac{d^2\chi_\delta^{(3)}}{dx_i^2} (x_i) \
        a_{33}^{(3)} \, \Pi_{k}^{(3)} (\xi)
    \Bigg)
    \Bigg|_{
        \xi_3 = \frac{\ell_3 - x_3}{\varepsilon}, \,
        \bar{\xi}_3 = \frac{\overline{x}_3}{\varepsilon}
    },
\end{multline}
\begin{equation}\label{R4}
\breve{R}_{\varepsilon, 4, (i)}^{(m)}(x)
 =- \varepsilon^{m + 1}
    \chi_{\delta, 0}^{(i)} (x_i)
    \Big(
        \big[w^{(i)}_m(x_i) + u^{(i)}_{m}(x_i, \bar{\xi}_i)\big] \, \overline{V}^{(i)}(x_i, \bar{\xi}_i) \cdot \bar{\nu}_{\bar{\xi}_i}
    \Big)
    \Big|_{\bar{\xi}_i = \frac{\overline{x}_i}{\varepsilon}}, \ \ i \in \{1, 2, 3\}.
\end{equation}

From~\eqref{R1} we conclude that
\begin{equation}\label{R1_estimate}
\exists\, \check{C}_m > 0  \ \ \exists \, \varepsilon_0 > 0 \ \
\forall\, \varepsilon\in (0, \varepsilon_0) : \quad
    \sup\limits_{x\in\Omega_{\varepsilon}}
    \left| R_{\varepsilon, 1, (i)}^{(m)}(x) \right|
    \leq
    \check{C}_m \varepsilon^{m}, \quad i \in \{1, 2, 3\}.
\end{equation}

Due to the functions
$\{ {N}_{k} - {w}_{k}^{(i)}(0), \ \Pi^{(i)}_{k} \}_{i=1}^{3}$
tend to zero as
$\xi_i \to +\infty, \, i\in\{1, 2, 3\}$
(see~\eqref{view_solution}) and the fact that the support of the derivatives of cut-off function
$\chi_{\delta, 0}^{(i)} (x_i)$
belongs to the set
$\{x_i: \varepsilon {\ell_0} + \delta \le
x_i \le \varepsilon {\ell_0} + 2 \delta\},$
we arrive that
\begin{equation}\label{R2_estimate}
    \sup\limits_{x\in\Omega_{\varepsilon}}
    \left| R_{\varepsilon, 2, (i)}^{(m)}(x) \right|
    \leq
    \check{C}_m
    \exp{ \left(
      - \frac{\varepsilon\ell_0 + \delta}{\varepsilon} \, \gamma_i
    \right) }, \quad i \in \{1, 2, 3\},
\end{equation}
similarly we obtain that
\begin{equation}\label{R3_estimate}
    \sup\limits_{x\in\Omega_{\varepsilon}}
    \left| R_{\varepsilon, 3, (3)}^{(m)}(x) \right|
    \leq
    \check{C}_m
    \exp{ \left(
       - \frac{\delta}{\varepsilon} \,
         \frac{v_3^{(3)}(\ell_3)}{a_{33}^{(3)}}
    \right) }.
\end{equation}

It follows from~(\ref{R4}) that there exist positive constants $\overline{C}_m$ and $\overline{\varepsilon}_0$ such that
\begin{equation}\label{R4_estimate}
\forall \, \varepsilon\in(0, \overline{\varepsilon}_0): \quad
    \sup\limits_{x\in\Gamma_\varepsilon^{(i)}}
    \left| \breve{R}_{\varepsilon, 4, (i)}^{(m)}(x) \right|
    \leq
    \overline{C}_m\varepsilon^{m + 1}, \quad i \in \{1, 2, 3\}.
\end{equation}

Using estimates~(\ref{R1_estimate})~--~(\ref{R4_estimate}) we obtain the following estimates:
\begin{gather}
\label{R1_norm_estimate}
    \left\| R_{\varepsilon, 1, (i)}^{(m)} \right\|_{L^2 (\Omega_\varepsilon)}
    \leq
    \check{C}_m \sqrt{\pi \ell_i \ } \ h_i \
    \varepsilon^{m + 1},
\quad i \in \{1, 2, 3\},
\\
\label{R2_norm_estimate}
    \left\| R_{\varepsilon, 2, (i)}^{(m)} \right\|_{L^2 (\Omega_\varepsilon)}
    \leq
    \check{C}_m \sqrt{\pi \ } \ h_i \ \delta^{\frac{1}{2}} \
    \exp{\left(
      - \frac{\varepsilon\ell_0 + \delta}{\varepsilon} \, \gamma_i
    \right)},
\quad i \in \{1, 2, 3\},
\\
\label{R3_norm_estimate}
    \left\| R_{\varepsilon, 3, (3)}^{(m)} \right\|_{L^2 (\Omega_\varepsilon)}
    \leq
    \check{C}_m \sqrt{\pi \ } \ h_i \ \delta^{\frac{1}{2}} \
    \varepsilon
    \exp{\left(
      - \frac{\delta}{\varepsilon} \
        \frac{v_3^{(3)}(\ell_3)}{a_{33}^{(3)}}
    \right)},
\\
\label{R4_norm_estimate}
    \left\|
        \breve{R}_{\varepsilon, 4, (i)}^{(m)}
    \right\|_{L^2 (\Gamma_\varepsilon^{(i)})}
    \leq
    \overline{C}_m \sqrt{2 \pi \ell_i h_i \ } \
    \varepsilon^{m + \frac32},
\quad i \in \{1, 2, 3\},
\end{gather}

Thus, the difference
$W_\varepsilon := u_\varepsilon - U_\varepsilon^{(m)}$
satisfies the following relations:
\begin{equation}\label{Residuals_Difference}
\left\{\begin{array}{rcll}
  - \varepsilon\, \mathrm{div} \big( \mathbb{D}^{(i)}_\varepsilon
    \nabla W_\varepsilon\big)
  + \mathrm{div}
    \big( \overrightarrow{V_\varepsilon}^{(i)} W_\varepsilon \big)
    & = & R^{(m)}_{\varepsilon, (i)} &
    \text{in} \ \Omega_\varepsilon^{(i)}, \ \ i \in \{1, 2, 3\},
\\[2mm]
  - \varepsilon \, \sigma_\varepsilon(W_\varepsilon)
  + \overrightarrow{V_\varepsilon}^{(i)} \cdot
    \boldsymbol{\nu} \, W_\varepsilon
    & = & \breve{R}_{\varepsilon, (i)}^{(m)} &
    \text{on} \ \Gamma_\varepsilon^{(i)}, \ \   i \in \{1, 2, 3\},
\\[2mm]
  - \varepsilon\, \mathrm{div} \big( \mathbb{D}^{(0)}_\varepsilon
    \nabla W_\varepsilon\big)
  + \mathrm{div}
    \big( \overrightarrow{V_\varepsilon}^{(0)} W_\varepsilon \big)
    & = & 0 &
    \text{in} \ \Omega_\varepsilon^{(0)},
\\[2mm]
  - \varepsilon \, \sigma_\varepsilon(W_\varepsilon)
  + \overrightarrow{V_\varepsilon}^{(0)} \cdot
    \boldsymbol{\nu} \, W_\varepsilon
    & = & 0 &
    \text{on} \ \Gamma_\varepsilon^{(0)},
\\[2mm]
    W_\varepsilon \big|_{x_i = \ell_i}
    & = & 0 &
    \text{on} \ \Upsilon_{\varepsilon}^{(i)} (\ell_i) , \ \ i \in \{1, 2, 3\},
 \\[2mm]
    W_\varepsilon\big|_{x_i = \varepsilon \ell_0-0}
    & = & W_\varepsilon\big|_{x_i = \varepsilon \ell_0+0}, & i \in \{1, 2, 3\},
 \\[2mm]
    a_{ii}^{(i)} \,
    \partial_{x_i}W_\varepsilon\big|_{x_i = \varepsilon \ell_0+0}
    & = &
    \sum_{j=1}^{3}
    \big(
        a_{ij}^{(0)}(\frac{x}{\varepsilon}) \partial_{x_j} W_\varepsilon
    \big)
    \big|_{x_i = \varepsilon \ell_0-0}, &
    i \in \{1, 2, 3\}.
\end{array}\right.
\end{equation}
This means that the series~\eqref{asymp_expansion} is a formal asymptotic solution to the problem~\eqref{probl}.
Considering the estimate \eqref{appriory_estimate}, we get  the asymptotic estimate~\eqref{main_estimate} and prove the theorem.
\end{proof}

\begin{corollary}\label{Corollary}
The differences between the solution $u_\varepsilon$ of problem~\eqref{probl} and the partial sum $U_\varepsilon^{(0)}$ $(\text{see}\ ~\eqref{aaN})$ admit the following asymptotic estimates:
\begin{equation}\label{aa1_H1_estimate}
    \| \,
        u_\varepsilon - U_\varepsilon^{(1)}
    \|_{H^1(\Omega_\varepsilon)}
    \leq
    \widetilde{C}_0 \, \varepsilon^{2},
\end{equation}
\begin{equation}\label{aa0_L2_estimate}
    \| \,
        u_\varepsilon - U_\varepsilon^{(0)}
    \|_{L^2(\Omega_\varepsilon)}
    \leq
    \widetilde{C}_0 \, \varepsilon^{2}.
\end{equation}
In thin cylinders
$
\Omega_{\varepsilon,\delta}^{(i)} :=
\Omega_\varepsilon^{(i)} \cap
\big\{
    x\in \Bbb{R}^3 : \ x_i\in I_{\varepsilon, \delta}^{(i)}
\big\}, \ i \in \{1, 2, 3\},
$
the following estimates hold:
\begin{equation}\label{w1_cylinder_H1_estimate}
    \| \,
        u_\varepsilon
        - w_0^{(i)} - \varepsilon w_1^{(i)} - \varepsilon u_1^{(i)}
    \|_{H^1(\Omega_{\varepsilon,\delta}^{(i)})}
    \leq
    \widetilde{C}_1 \, \varepsilon^2, \ \ i \in \{1, 2, 3\},
\end{equation}
in addition,
\begin{equation}\label{w1_segment_H1_estimate}
    \| \,
        E^{(i)}_\varepsilon(u_\varepsilon)
        -
        w_0^{(i)} - \varepsilon w_1^{(i)}
        -
        \varepsilon E^{(i)}_\varepsilon(u_1^{(i)})
    \|_{H^1(I_{\varepsilon, \delta}^{(i)})}
    \leq
    \widetilde{C}_2 \, \varepsilon, \ \ i \in \{1, 2, 3\},
\end{equation}
\begin{equation}\label{w1_segment_C_estimate}
    \max_{x_i \in \overline{I_{\varepsilon, \delta}^{(i)}}}
    | \,
        E^{(i)}_\varepsilon(u_\varepsilon)(x_i)
        -
        w_0^{(i)}(x_i) - \varepsilon w_1^{(i)}(x_i)
        -
        \varepsilon E^{(i)}_\varepsilon(u_1^{(i)})(x_i)
    |
    \leq
    \widetilde{C}_2 \, \varepsilon, \ \ i \in \{1, 2, 3\},
\end{equation}
where
$
I_{\varepsilon, \delta}^{(i)} :=
(\varepsilon \ell_0 + 2 \delta, \ell_i - 2 \delta), \ i \in \{1, 2, 3\}
$
and
\begin{equation*}
    E^{(i)}_\varepsilon(u_\varepsilon)(x_i)
  = \frac{1}{\pi \varepsilon^2\, h_i^2}
    \int_{\Upsilon^{(i)}_\varepsilon}
    u_\varepsilon(x)\, d\overline{x}_i,
\quad i=1,2,3.
\end{equation*}

In the neighbourhood
$
\Omega^{(0)}_{\varepsilon, \delta} :=
\Omega_\varepsilon \cap
\big\{ x\in\Bbb{R}^3 : \ x_i<\ell_0\varepsilon + \delta, \ i=1,2,3 \big\}
$
of the joint, we get estimates
\begin{equation}\label{N1_joint_H1_estimate}
        \| \,
        u_\varepsilon - N_0 - \varepsilon N_1
    \|_{H^1(\Omega^{(0)}_{\varepsilon, \delta})}
    \leq
    \widetilde{C}_3 \, \varepsilon^{\frac52}.
\end{equation}

In the neighbourhood
$
\Omega^{(3)}_{\varepsilon, \delta, \ell_3} :=
\Omega_\varepsilon^{(3)} \cap
\big\{ x\in\Bbb{R}^3 : \ x_3 > \ell_3 - \delta \big\}
$
of the base $\Upsilon_{\varepsilon}^{(3)} (\ell_3),$ we have
\begin{equation}\label{P1_base_H1_estimate}
    \| \,
        u_\varepsilon - w_0^{(3)} - \Pi_0^{(3)}
        -
        \varepsilon w_1^{(3)} - \varepsilon u_1^{(3)} - \varepsilon \Pi_1^{(3)}
    \|_{H^1(\Omega^{(0)}_{\varepsilon, \delta, \ell_3})}
    \leq
    \widetilde{C}_3 \, \varepsilon^{2} \delta^\frac12.
\end{equation}
\end{corollary}

\begin{proof}
Using the smoothness of the functions $\{w_k^{(i)}\}$ and the exponential decay of  the functions $\{N_k~-~w_k^{(i)}(0)\}$ and
$\{\Pi_{k}^{(i)}\}, \ i \in \{1, 2, 3\},$
at infinity, we deduce the inequality~\eqref{aa1_H1_estimate} from  estimate~\eqref{main_estimate} at $m = 2$:
\begin{gather*}
    \left\| \,
        u_\varepsilon - U^{(1)}_{\varepsilon}
    \right\|_{H^1(\Omega_\varepsilon)}
    \le
    \left\| \,
        u_\varepsilon - U^{(2)}_{\varepsilon}
    \right\|_{H^1(\Omega_\varepsilon)}
\\
    + \,
    \Bigg\| \,
    \varepsilon^2
    \left(
        \sum_{i=1}^3
        \Big(
            \chi_{\delta, 0}^{(i)} \,(w_{2}^{(i)} + u_{2}^{(i)})
        \Big)
      + \bigg( 1 - \sum_{i=1}^3 \chi_{\delta, 0}^{(i)} \, \bigg) {N}_{2}
      + \chi_{\delta}^{(3)} \Pi_{2}^{(3)}
    \right)
    \Bigg\|_{H^1(\Omega_\varepsilon)}
\\
    \le
    C_2 \, \varepsilon^{2}
    +
    \varepsilon^2
    \sum\limits_{i=1}^3
    \bigg\| \,
        \left(
            \chi_{\delta, 0}^{(i)} (w_{2}^{(i)} + u_{2}^{(i)})
          + \big( 1 - \chi_{\delta, 0}^{(i)} \big) N_2
        \right)
    \bigg\|_{H^1(\Omega^{(i)}_\varepsilon)}
\\
    + \,
    \varepsilon^2
    \left\| N_2 \right\|_{H^1(\Omega_\varepsilon^{(0)})}
  + \varepsilon^2
    \left\| \,
        \chi_\delta^{(3)} \Pi_2^{(3)}
    \right\|_{H^1(\Omega_\varepsilon^{(i)})}
\\
    \le
    C_2 \, \varepsilon^{2}
    +
    \varepsilon^2
    \sum_{i=1}^3
    \left\| \,
        w_2^{(i)} + u_2^{(i)}
    \right\|_{H^1(\Omega_\varepsilon^{(i)})}
    +
    \varepsilon^2
    \sum_{i=1}^3
    \left\| \,
        \big(1 - \chi_{\delta, 0}^{(i)}\big) \,
        ( N_2 - w_2^{(i)}(0) )
    \right\|_{H^1(\Omega_\varepsilon^{(i)})}
\\
    + \,
    \varepsilon^\frac52
    \left\| N_2 \right\|_{H^1(\Xi^{(0)})}
    +
    \varepsilon^2
    \left\| \,
        \chi_\delta^{(3)} \Pi_2^{(3)}
    \right\|_{H^1(\Omega_\varepsilon^{(3)})}
    \le
    \widetilde{C}_0 \, \varepsilon^2.
\end{gather*}
The estimate~\eqref{aa0_L2_estimate} can be proven similarly with the help of~\eqref{main_estimate} at $m = 2$ and the $L^2$-norm of corresponding terms.

Again with the help of estimate~\eqref{main_estimate} at $m=2,$ we deduce
\begin{equation*}
    \left\| \,
        u_\varepsilon
        -
        w_0^{(i)}  - \varepsilon w_1^{(i)} - \varepsilon u_1^{(i)}
    \right\|_{H^1(\Omega_{\varepsilon,\delta}^{(i)})}
    \leq
    \left\| \,
        u_\varepsilon - U^{(2)}_{\varepsilon}
    \right\|_{H^1(\Omega_\varepsilon)}
    +
    \varepsilon^2
    \left\| \,
        w_{2}^{(i)} + u_{2}^{(i)}
    \right\|_{H^1(\Omega_{\varepsilon,\delta}^{(i)})}
    \leq
    \widetilde{C}_1 \, \varepsilon^2,
\end{equation*}
whence we get~\eqref{w1_cylinder_H1_estimate}.

Using the Cauchy-Buniakovskii-Schwarz inequality and~\eqref{w1_cylinder_H1_estimate}, we obtain inequalities~\eqref{w1_segment_H1_estimate}.
Since the space $H^1(I_{\varepsilon, \delta}^{(i)})$ continuously embedded in
$C(\overline{I}_{\varepsilon, \delta}^{(i)}),$
from~\eqref{w1_segment_H1_estimate} it follows inequalities~\eqref{w1_segment_C_estimate}.

From inequalities
\begin{equation*}
    \left\| \,
        u_\varepsilon - N_0 - \varepsilon N_1
    \right\|_{H^1(\Omega^{(0)}_{\varepsilon, \delta})}
    \leq
    \left\| \,
        u_\varepsilon - U^{(3)}_{\varepsilon}
    \right\|_{H^1(\Omega_\varepsilon)}
    +
    \varepsilon^2 \,
    \| N_2 \|_{H^1(\Omega^{(0)}_{\varepsilon, \delta})}
    +
    \varepsilon^3 \,
    \| N_3 \|_{H^1(\Omega^{(0)}_{\varepsilon, \delta})}
    \leq
    \widetilde{C}_3 \, \varepsilon^{\frac52},
\end{equation*}
\begin{gather*}
    \left\| \,
        u_\varepsilon - w_0^{(3)} - \Pi_0^{(3)}
        -
        \varepsilon w_1^{(3)} - \varepsilon u_1^{(3)} - \varepsilon \Pi_1^{(3)}
    \right\|_{H^1(\Omega_{\varepsilon,\delta,\ell_3}^{(3)})}
    \leq
    \left\| \,
        u_\varepsilon - U^{(3)}_{\varepsilon}
    \right\|_{H^1(\Omega_\varepsilon)}
\\
    +
    \,\varepsilon^2
    \left\| \,
        w_{2}^{(3)} + u_{2}^{(3)} + \Pi_{2}^{(3)}
    \right\|_{H^1(\Omega_{\varepsilon,\delta,\ell_3}^{(3)})}
    +
    \,\varepsilon^3
    \left\| \,
        w_{3}^{(3)} + u_{3}^{(3)} + \Pi_{3}^{(3)}
    \right\|_{H^1(\Omega_{\varepsilon,\delta,\ell_3}^{(3)})}
    \leq
    \widetilde{C}_3 \, \varepsilon^{2} \delta^\frac12
\end{gather*}
it follows more better energetic estimates~\eqref{N1_joint_H1_estimate}, \eqref{P1_base_H1_estimate} in a neighbourhood of the joint $\Omega^{(0)}_\varepsilon$ and in a neighbourhood of the base $\Upsilon_{\varepsilon}^{(3)} (\ell_3).$
\end{proof}

\section{Conclusion}\label{conclusion}

{\bf 1.}
In this paper, we have presented asymptotic analysis of  transport phenomena in non-homogeneous environments that can fill thin cracks, channels and junctions between them. In Sec.~\ref{Sec:expansions}, formal asymptotic expansions are constructed under general classical assumptions for the velocity field $\overrightarrow{V_\varepsilon}.$ Here we wanted also to study the influence of surface transport phenomena on the lateral  surfaces of the thin cylinders, in which the flows of both the convective and diffusion parts are involved. This influence is taken into account through the terms $\{u^{(i)}_k\}$ of the regular asymptotics, which are the solutions of the problems \eqref{problem_u_k}, respectively. The impact also manifests itself in the limit problem \eqref{limit_prob} through the given functions $\{\varphi^{(i)}\}$ (the right-hand sides in the Neumann type conditions in the problem \eqref{probl}).

The local geometric irregularity of the node and physical processes inside do not affect the view of the limit problem. However, these ones are taken into account by the terms of  the inner asymptotics, which are solutions to the problems \eqref{N_k_prob}.

Due to a small diffusion the limit problem \eqref{limit_prob} consists of three first order differential equations on the graph $\mathcal{I}:= \bigcup_{i=1}^3 I_i$ with the special Kirchhoff condition at the vertex and with only two boundary conditions at the corresponding ends of this graph. This is in complete agreement with asymptotic theory of boundary value problems with a small parameter at the highest derivatives (in the limit, such problems turn into problems of a lower order). As simple examples show (see e.g. \cite[\S 1.1]{Stynes-2018}, where one-dimensional convection-dominated problem was considered) such solutions
have boundary layers near some parts of the domain boundary. In our case it is the base of the thin cylinder $\Omega_\varepsilon^{(3)}.$ Therefore, we introduce the boundary-layer expansion (see Subsec.~\ref{boundary-layer}) in a neighborhood of the base $\Upsilon_{\varepsilon}^{(3)} (\ell_3)$ in order to satisfy the  boundary condition  and to neutralize the discrepancies from the regular part of the asymptotics.

\medskip

{\bf 2.}
An important task of each proposed asymptotic approach is its stability and accuracy. This is provided by proving a priori estimates for solutions. In addition, such estimates help to obtain the error estimate between the constructed approximation and the exact solution and hence to analyse effectiveness of the proposed asymptotic method (or numerical one).

Proving such estimates for convection-dominated problems can demand additional restrictions on coefficients of the equation (see e.g.
\cite[Lemma 4.14]{Stynes-2018}, where an a priori estimate for two-dimensional problem with the Dirichlet boundary condition in the weighted energy norm was proved). For our problem we additionally imposed conditions~\eqref{assum_1}--\eqref{potential} for the velocity field $\overrightarrow{V_\varepsilon}.$
The conditions \eqref{assum_1} are restrictions on the components of $\overrightarrow{V_\varepsilon}$ in the axial directions of the thin cylinders and mean that the convective terms really dominate significantly in these  directions.
At the same time, it follows from \eqref{str_1} - \eqref{st_3} that the diffusion terms and convection ones are the same order  $\mathcal{O}(\varepsilon)$ in the perpendicular direction to the lateral surfaces of thin cylinders, and
the conditions \eqref{assum_2} mean that diffusion processes prevail in this direction. In the node $\Omega_\varepsilon^{(0)}$ the velocity field $\overrightarrow{V_\varepsilon}$ is conservative and incompressible and \eqref{cond_1} signifies that the flow entering the disk $\Upsilon_\varepsilon^{(1)} (\varepsilon\ell_0)$   of the node boundary from the first thin cylinder is equal to the sum of flows outgoing from the second and third disks into the other thin cylinders.

\medskip

{\bf 3.} The results obtained in Theorem~\ref{mainTheorem} and Corollary~\ref{Corollary} show that
convection-diffusion problems on graphs do not correctly model real processes in thin graph-like junctions, because they do not take into account local geometric heterogeneities of nodes and physical processes inside, and the behaviour of solutions near  bases of some thin cylinders. It means that  the perturbed boundary-value problem \eqref{probl}  cannot be replaced by the corresponding 1-dimensional limit  problem \eqref{limit_prob} on the graph from physical point of view.

The estimate \eqref{aa0_L2_estimate} shows that the zero approximation
$$
U_\varepsilon^{(0)}(x) = \overline{u}_0(x, \,\varepsilon) + \overline{N}_0 (x, \,\varepsilon) + \overline{\Pi}_0(x, \,\varepsilon), \quad x \in \Omega_\varepsilon,
$$
contains terms from both the internal  and  boundary-layer asymptotics. Thus, the main our conclusion is as follows:

\emph{In parallel with the limit problem,  the problems  for the first terms $N_0$ and $\Pi_0^{(3)}$
of the inner and boundary-layer  asymptotics should be necessarily considered.
}

\medskip

{\bf 4.} For convection-diffusion problems it is very important to know the behaviour of the solution gradient.
Thanks to estimates~\eqref{w1_cylinder_H1_estimate}, \eqref{P1_base_H1_estimate} and~\eqref{N1_joint_H1_estimate}, we get the zero-order approximation of the gradient (flux) of the solution
$$
\nabla u_\varepsilon(x)
\sim
\bigg(
    \frac{\partial w^{(i)}_0}{\partial x_i}(x_i)
    +
    \varepsilon \frac{\partial w^{(i)}_1}{\partial x_i}(x_i)
    +
    \varepsilon \frac{\partial u^{(i)}_1}{\partial x_i}
    \Big(x_i, \frac{\overline{x}_i}{\varepsilon}\Big), \
    \nabla_{\bar{\xi}_i} u^{(i)}_1(x_i, \bar{\xi}_i)
    |_{\bar{\xi}_i=\frac{\overline{x}_i}{\varepsilon}}
\bigg)
\quad \text{as}\quad \varepsilon \to 0
$$
in each cylinder~$\Omega_{\varepsilon, \delta}^{(i)},$ $i\in \{1,2,3\},$
$$
\nabla u_\varepsilon(x)
\sim
\varepsilon^{-1} \nabla_{\xi}{N}_{0}(\xi)\Big|_{\xi=\frac{x}{\varepsilon}}  + \nabla_{\xi}{N}_{1}(\xi)\Big|_{\xi=\frac{x}{\varepsilon}}
\quad \text{as}\quad \varepsilon \to 0
$$
in the neighbourhood $\Omega^{(0)}_{\varepsilon, \delta}$ of the node, and
\begin{gather*}
\nabla u_\varepsilon(x)
\sim
\Bigg(
    \frac{\partial w^{(3)}_0}{\partial x_3}(x_3)
    -
    \varepsilon^{-1}
    \frac{\partial \Pi_0^{(3)}}{\partial \xi_3}
    \Big(
        \frac{x_1}{\varepsilon}, \frac{x_2}{\varepsilon}, \xi_3
    \Big)\Big|_{\xi_3 = \frac{\ell_3 - x_3}{\varepsilon}}
    +
\\
    \varepsilon \frac{\partial w^{(3)}_1}{\partial x_i}(x_i)
    +
    \varepsilon \frac{\partial u^{(3)}_1}{\partial x_i}
    \Big(x_3, \frac{\overline{x}_3}{\varepsilon}\Big)
    -
    \frac{\partial \Pi_1^{(3)}}{\partial \xi_3}
    \Big(
        \frac{x_1}{\varepsilon}, \frac{x_2}{\varepsilon}, \xi_3
    \Big)\Big|_{\xi_3 = \frac{\ell_3 - x_3}{\varepsilon}},
\\
    \varepsilon^{-1}
    \nabla_{\bar{\xi}_3} \Pi_0^{(3)}
    \Big(
        \frac{\ell_3 - x_3}{\varepsilon}, \bar{\xi}_3
    \Big)\Big|_{\bar{\xi}_3=\frac{\overline{x}_3}{\varepsilon}}
    +
    \nabla_{\bar{\xi}_3} u^{(3)}_1(x_3, \bar{\xi}_3)
    \Big|_{\bar{\xi}_3=\frac{\overline{x}_3}{\varepsilon}}
    +
    \nabla_{\bar{\xi}_3} \Pi_1^{(3)}
    \Big(
        \frac{\ell_3 - x_3}{\varepsilon}, \bar{\xi}_3
    \Big)\Big|_{\bar{\xi}_3=\frac{\overline{x}_3}{\varepsilon}}
\Bigg)
\quad \text{as}\quad \varepsilon \to 0
\end{gather*}
in the neighbourhood $\Omega^{(3)}_{\varepsilon, \delta, \ell_3}$ of the base $\Upsilon_{\varepsilon}^{(3)} (\ell_3)$.

Also, finding the concentration of a substance at each point is especially important for applied problems;
just the uniform pointwise estimates ~\eqref{w1_segment_C_estimate}   show how to find it without knowing the exact solution.

\medskip

{\bf 5.} The method proposed in the present paper can be used for the asymptotic investigation of convection-diffusion problems
in thin graph-like junctions with more complex structures (e.g., variable thickness of thin curvilinear
cylinders, perforated thin cylinders).  It would also be very interesting to consider the case when the velocity field $\overrightarrow{V_\varepsilon}$ is not conservative, or incompressible in the node.

\section*{Acknowledgements}
The first author is very grateful to Professor Christian Rohde for bringing the author attention to convection-diffusion problems in thin graph-like junctions, as well as for the hospitality and excellent working conditions at  Stuttgart University, where this study was started in January 2020  under the support of  the SFB 1313 center of the University of Stuttgart.

This work has been supported by the Grant of the Ministry of Education and Science of Ukraine for
perspective development of a scientific direction "Mathematical sciences and natural sciences" at Taras Shevchenko National University
of Kyiv.

\end{document}